\documentclass[12pt]{amsart}
\usepackage{amsmath}
\usepackage{amssymb,latexsym,amsthm,amsfonts,amscd,pgf,comment,enumerate,ragged2e}
\usepackage{geometry}
\usepackage{tikz,amsmath}
\usepackage{subcaption}
\usepackage{graphicx}
\usepackage{float}
\usepackage{ragged2e}
\usepackage[pdftex]{hyperref}
\newtheorem{theorem}{Theorem}[section]
\newtheorem{lemma}[theorem]{Lemma}
\newtheorem{proposition}[theorem]{Proposition}

\newtheorem{remark}[theorem]{Remark}

\newtheorem{definition}[theorem]{Definition}

\numberwithin{equation}{section}

\begin{document}

\baselineskip=15.5pt

\title[Integer Sequences and Monomial Ideals]{Integer Sequences and Monomial Ideals}

\author[C. Kumar]{Chanchal Kumar}

\address{IISER Mohali,
Knowledge City, Sector 81, SAS Nagar, Punjab -140 306, India.}

\email{chanchal@iisermohali.ac.in}

\author[A. Roy]{Amit Roy}

\address{IISER Mohali,
Knowledge City, Sector 81, SAS Nagar, Punjab -140 306, India.}

\email{amitroy@iisermohali.ac.in}


\date{}

\maketitle

\begin{abstract}
Let $\mathfrak{S}_n$ be the set of all permutations of 
$[n]=\{1,\ldots,n\}$ and 
let $W$ be the subset consisting of permutations
$\sigma \in \mathfrak{S}_n$ avoiding 132 and 312-patterns. The monomial ideal $I_W =
\left\langle \mathbf{x}^{\sigma} = 
\prod_{i=1}^n x_i^{\sigma(i)} : 
\sigma \in W \right\rangle $ in the 
polynomial ring $R = k[x_1,\ldots,x_n]$ over a 
field $k$ is called a {\em hypercubic ideal} in \cite{AC1}. The Alexander 
dual $I_W^{[\mathbf{n}]}$ of $I_W$ with respect to 
$\mathbf{n}=(n,\ldots,n)$ has the minimal cellular resolution 
supported on the first barycentric 
subdivision $\mathbf{Bd}(\Delta_{n-1})$ of an 
$n-1$-simplex $\Delta_{n-1}$. We show that 
the number of standard monomials of the 
Artinian quotient $\frac{R}{I_W^{[\mathbf{n}]}}$ equals
the number of rooted-labelled  unimodal forests on the
vertex set $[n]$. In other words,
 \[ \dim_k\left(\frac{R}{I_W^{[\mathbf{n}]}}\right) = 
 \sum_{r=1}^n r!~s(n,r) =  {\rm Per}\left([m_{ij}]_{n \times n} \right),\]
where $s(n,r)$ is the (signless) Stirling number of
 the first kind and ${\rm Per}([m_{ij}]_{n \times n})$ 
 is the permanent of the matrix $[m_{ij}]$ with $m_{ii}=i$ and $m_{ij}=1$ for $i \ne j$.
For various subsets $S$ of $\mathfrak{S}_n$ consisting of 
permutations avoiding patterns, the corresponding integer
sequences $\left\lbrace \dim_k\left(\frac{R}{I_S^{[\mathbf{n}]}}\right) 
\right\rbrace_{n=1}^{\infty}$ are identified.\\

\end{abstract}

\section{Introduction}
Let $G$ be an oriented graph on the vertex set 
$\{0,1,\ldots,n\}$ rooted at ${0}$. A nonoriented graph
on $\{0,1,\ldots,n\}$ has the symmetric adjacency matrix and 
it is identified with a unique rooted oriented graph on 
$\{0,1,\ldots,n\}$ having the same (symmetric) adjacency matrix.
Let $R=k[x_1,\ldots,x_n]$ be the standard polynomial ring 
in $n$ variables over a field $k$.
Postnikov and Shapiro \cite{PoSh} associated a monomial ideal 
${\mathcal{M}}_G$ in $R$ such that 
the number of standard monomials of
the Artinian quotient
$\frac{R}{{\mathcal{M}}_G}$ is precisely the number of 
oriented-spanning trees of $G$. A sequence 
$\mathbf{p}=(p_1,\ldots,p_n) \in \mathbb{N}^n$ is called a 
{\em $G$-parking function} if $\mathbf{x}^{\mathbf{p}}=
\prod_{i=1}^{n} x_i^{p_i}$ is a standard monomial of 
$\frac{R}{\mathcal{M}_G}$ (i.e., $\mathbf{x}^{\mathbf{p}} \notin 
\mathcal{M}_G$). Let ${\rm SPT}(G)$ be the set of (oriented) 
spanning trees of $G$ rooted at $0$ and ${\rm PF}(G)$ be the set
of $G$-parking functions of $G$. Then $|{\rm PF}(G)|=
|{\rm SPT}(G)|$ (see \cite{PoSh}).

If $G$ is the complete graph $K_{n+1}$ on 
the vertex set 
$\{0,1,\ldots,n\}$, then 
\[ \mathcal{M}_{K_{n+1}}= \left\langle \left( 
\prod_{i \in I} x_i \right)^{n-|I|+1} : \emptyset \ne I 
\subseteq [n] \right\rangle
\]
is called a {\em tree ideal}. Cayley's formula for enumeration of
labelled trees states that $|{\rm SPT}(K_{n+1})|=(n+1)^{n-1}$.
Also the set ${\rm PF}(K_{n+1})$ of
 $K_{n+1}$-parking functions is the set ${\rm PF}_n$ of (ordinary) parking functions
of length $n$. A finite sequence $\mathbf{p}=(p_1,\ldots,p_n)
\in \mathbb{N}^n$
with $0 \le p_i < n$ is called a {\em parking function} of
length $n$ if a nondecreasing rearrangement $p_{i_1} \le p_{i_2}
\le \ldots \le p_{i_n}$ of $\mathbf{p}$ satisfies $p_{i_j} < j$ 
for $1 \le j \le n$. A recursively defined bijection $\phi :
{\rm PF}_n \longrightarrow {\rm SPT}(K_{n+1})$ has been 
constructed by Kreweras \cite{K}. Parking functions 
or more generally, vector parking functions have appeared in many
areas of mathematics. For more on parking functions, we
refer to \cite{PiSt,Ya}. An algorithmic bijection
$\phi : {\rm PF}(G) \longrightarrow {\rm SPT}(G)$, called 
{\em DFS-burning algorithm}, is given by 
Perkinsons et. al. \cite{PYY}
for a simple graph $G$ and by Gaydarov and Hopkins \cite{GH}  for
 multigraph $G$.

Let $\mathfrak{S}_n$ be the set of all permutations of $[n]=
\{1,2,\ldots,n\}$. For 
$r \le n$, consider a $\tau \in \mathfrak{S}_r$, called a 
{\em pattern}. 
A permutation $\sigma \in \mathfrak{S}_n$ is said to {\em 
avoid a pattern} $\tau$ if there is no subsequence in
$\sigma=\sigma(1)\sigma(2) \ldots \sigma(n)$ that is in the same 
relative order as $\tau$. Let $\mathfrak{S}_n(\tau)$ be the
subset consisting of permutations $\sigma \in \mathfrak{S}_n$
that avoid pattern $\tau$. If $r > n$, then 
$\mathfrak{S}_n(\tau) = \mathfrak{S}_n$.  Also,
if $\tau^{(i)} \in \mathfrak{S}_{r_i}$ for $1\le i \le s$, then 
$\mathfrak{S}_n(\tau^{(1)},\ldots,\tau^{(s)}) = \bigcap_{j=1}^s 
\mathfrak{S}_n(\tau^{(j)})$.  
Enumeration and combinatorial properties of the set of permutations avoiding
patterns are obtained in \cite{SiSc}.

For a nonempty subset $S \subseteq \mathfrak{S}_n$, 
consider the monomial ideal 
$I_S = \langle {\mathbf{x}}^{\sigma} = \prod_{i=1}^n x_i^{\sigma(i)}
: \sigma \in S \rangle $  in $R=k[x_1,\ldots,x_n]$ induced by 
$S$. The monomial ideal $I_{\mathfrak{S}_n}$ is called a
{\em permotuhedron ideal} and the Alexander dual
$I_{\mathfrak{S}_n}^{[\mathbf{n}]}$ is the tree ideal 
$\mathcal{M}_{K_{n+1}}$. 
The $i^{th}$ Betti number 
$\beta_i(I_{\mathfrak{S}_n}^{[\mathbf{n}]})$ of
$I_{\mathfrak{S}_n}^{[\mathbf{n}]}$ is given by
\[
\beta_i(I_{\mathfrak{S}_n}^{[\mathbf{n}]})=
\beta_{i+1}\left(\frac{R}{I_{\mathfrak{S}_n}^{[\mathbf{n}]}}\right)
= (i!) S(n+1,i+1); \quad (0 \le i \le n-1),
\]
where $S(n,r)$ is the Stirling number of the 
second kind, i.e., the
number of set-partitions of $[n]$ into $r$ blocks
(see \cite{PoSh}).
Further, we have already observed that the standard monomials
of $\frac{R}{I_{\mathfrak{S}_n}^{[\mathbf{n}]}}$ is given by
$\dim_k\left(\frac{R}{I_{\mathfrak{S}_n}^{[\mathbf{n}]}}\right)
= |{\rm PF}_n|=(n+1)^{n-1}$.

For various subsets 
$S \subseteq \mathfrak{S}_n$, 
the Alexander dual $I_S^{[\mathbf{n}]}$ of $I_S$ 
with respect to
$\mathbf{n}=(n,\ldots,n)$ has many interesting 
properties similar to
the Alexander dual of permutohedron ideal. The Betti numbers
and enumeration of standard monomials of the Alexander dual
$I_S^{[\mathbf{n}]}$ for subsets $S=\mathfrak{S}_n(132,231)$,
$\mathfrak{S}_n(123,132)$ and $\mathfrak{S}_n(123,132,213)$
are obtained in \cite{AC2, AC3}

Let $W=\mathfrak{S}_n(132,312)$. The monomial ideal 
$I_W$ of $R$
is called a {\em hypercubic ideal} in \cite{AC1}. The standard
monomials of $\frac{R}{I_W^{[\mathbf{n}]}}$ correspond 
bijectively to a subset $\widetilde{{\rm PF}}_n$ of ${\rm PF}_n$.
An element $\mathbf{p} \in \widetilde{{\rm PF}}_n$ is called a 
{\em restricted parking function} of length $n$. We show that the 
number of restricted parking functions of length $n$ is given by 
\[
\dim_k\left(\frac{R}{I_W^{[\mathbf{n}]}}\right) = 
|\widetilde{{\rm PF}}_n| = \sum_{r=1}^n (r!) ~s(n,r),
\]
where $s(n,r)$ is the (signless) Stirling number of 
the first kind, i.e., the
number of permutations of $[n]$ having exactly $r$ cycles in 
its cyclic decomposition. 
Thus the $n$th term of integer sequence
(A007840) in OEIS \cite{S} can be interpreted as the number of 
restricted parking functions of length $n$, 
or equivalently, as
the number of standard monomials of 
the Artinian quotient
$\frac{R}{I_W^{[\mathbf{n}]}}$.

The concept of pattern avoiding permutations has been generalized
to many combinatorial objects. A notion of rooted forests
that avoids a set of permutations is introduced and many classes
of such objects are enumerated in \cite{AnAr}. Let $F_n$ be the
set of rooted-labelled forests on $[n]$. Let 
$F_n(\tau)$ (or more generally, $F_n(\tau^{(1)},\ldots,\tau^{(r)})$) be the 
subset of $F_n$ consisting of
rooted-labelled forests avoiding a pattern $\tau$ (or
a set of patterns $\{ \tau^{(1)},\ldots,\tau^{(r)} \}$). We have
\[ 
|F_n(213,312)|= \sum_{r=1}^n (r!) ~s(n,r) = 
|\widetilde{{\rm PF}}_n|. 
\]
It is surprising that enumeration
of standard monomials of $\frac{R}{I_W^{[\mathbf{n}]}}$ and
enumeration of rooted-labelled forests $F_n(213,312)$ avoiding
213 and 312-patterns are related. It is an interesting problem to 
construct an algorithmic bijection
$\phi : \widetilde{{\rm PF}}_n \longrightarrow F_n(213,312)$,
analogous to DFS-burning algorithm that could explain the 
relationship between these objects.

The monomial ideal $I_S$ for many other subsets 
$S \subseteq \mathfrak{S}_n$, consisting of permutations
avoiding patterns are considered in the last section.

\section{Hypercubic ideals and restricted Parking functions}
Consider the subset $W=\mathfrak{S}_n(132,312)$ of permutations
of $[n]$ that avoid 132 and 312-patterns. 
For $\sigma \in \mathfrak{S}_n$,
it can be easily checked that 
$\sigma \in W$ if and only if 
$\sigma(1) \in [n]$ is arbitrary,
and $\sigma(j) = \ell$ for $j > 1$ if either
$\sigma(i) =  \ell+1$ or $\sigma(i) = \ell -1$ for
some $ i < j$. Clearly, $|W|= 2^{n-1}$. The monomial ideal 
$I_W$ appeared in \cite{AC1}, where it is called a
{\em hypercubic ideal}. Many properties of $I_W$ and its
Alexander dual $I_W^{[\mathbf{n}]}$ with respect to
$\mathbf{n}=(n,\ldots,n) \in \mathbb{N}^n$ 
have been obtained in \cite{AC1}. We proceed to enumerate 
the standard monomials of $\frac{R}{I_W^{[\mathbf{n}]}}$. For 
this purpose, we consider a little generalization.

Let $\mathbf{u}=(u_1,\ldots,u_n) \in \mathbb{N}^n$ with
$1 \le u_1 < u_2 < \ldots < u_n$. For $\sigma \in 
\mathfrak{S}_n$, let $\sigma\mathbf{u}=(u_{\sigma(1)},\ldots,
u_{\sigma(n)})$ and 
$\mathbf{x}^{\sigma\mathbf{u}}=
\prod_{i=1}^n x_i^{u_{\sigma(i)}}$. For any nonempty subset
$S \subseteq \mathfrak{S}_n$, we consider the monomial
ideal $I_S(\mathbf{u})= \left\langle\mathbf{x}^{\sigma\mathbf{u}}
: \sigma \in S \right\rangle$ in the polynomial ring
$R=k[x_1,\ldots,x_n]$. Clearly,
$I_S((1,2,\ldots,n)) =I_S$. 
The ideals
$I_{\mathfrak{S}_n}(\mathbf{u})$ and $I_W(\mathbf{u})$
are also called a {\em permutohedron ideal } and 
 a {\em hypercubic ideal}, respectively.
For an integer $c \ge 1$, we consider the 
Alexander dual $I_W(\mathbf{u})^{[\mathbf{u_n+c-1}]}$
of the hypercubic ideal $I_W(\mathbf{u})$ with respect to 
$\mathbf{u_n+c-1}=(u_n+c-1,\ldots,u_n+c-1) \in \mathbb{N}^n$.

\begin{proposition}
The minimal generators of 
$I_W(\mathbf{u})^{[\mathbf{u_n+c-1}]}$ are given by
\[
I_W(\mathbf{u})^{[\mathbf{u_n+c-1}]} = \left\langle 
\prod_{j \in T} x_j^{\mu_{j,T}^{\mathbf{u}}} : \emptyset
\ne T =\{j_1,\ldots,j_t\} \subseteq [n]; j_1 < \ldots < j_t \right\rangle,
\]
where 
$\mu_{j_1,T}^{\mathbf{u}} = u_n-u_t+c$ and
$\mu_{j_i,T}^{\mathbf{u}} = u_n-u_{t+j_i-i}+c $
for $2 \le i \le t$.
\label{P1}
\end{proposition}
\begin{proof}
The minimal generators of $I_W(\mathbf{u})^{[\mathbf{u_n}]}$ 
are
given in Theorem 3.3 of \cite{AC1}. Just replace 
$[\mathbf{u_n}]$ by $[\mathbf{u_n+c-1}]$. \hfill $\square$
\end{proof}

The Alexander dual 
$I_{\mathfrak{S}_n}(\mathbf{u})^{[\mathbf{u_n+c-1}]}$ of
the permutohedron ideal 
$I_{\mathfrak{S}_n}(\mathbf{u})$ is given by 
\[
I_{\mathfrak{S}_n}(\mathbf{u})^{[\mathbf{u_n+c-1}]}
= \left\langle 
\left(\prod_{j \in T} x_j\right)^{u_n-u_{|T|}+c} : 
 T \in \Sigma_n \right\rangle,
\]
where $\Sigma_n$ is the poset of all nonempty subsets of
$[n]$ ordered by inclusion.
Postnikov and Shapiro \cite{PoSh} showed that the monomial 
ideal $I_{\mathfrak{S}_n}(\mathbf{u})^{[\mathbf{u_n+c-1}]}$
is an order monomial ideal. Moreover, the minimal resolution
of $I_{\mathfrak{S}_n}(\mathbf{u})^{[\mathbf{u_n+c-1}]}$ is
the cellular resolution supported on the order complex
$\Delta(\Sigma_n)$ of $\Sigma_n$. Thus, the $i^{th}$ Betti number
\[
\beta_i(I_{\mathfrak{S}_n}(\mathbf{u})^{[\mathbf{u_n+c-1}]})
= (i!) S(n+1,i+1); \quad (0 \le i \le n-1), 
\]
where 
$S(n+1,i+1)$ is the Stirling number of the second kind. Further,
standard monomials of 
$\frac{R}{I_{\mathfrak{S}_n}(\mathbf{u})^{[\mathbf{u_n+c-1}]}}$ 
are given in terms of
$\lambda$-parking functions. Let $\lambda = (\lambda_1,\ldots,
\lambda_n)$ with $\lambda_i = u_n - u_i + c$. A sequence
$\mathbf{p}=(p_1,\ldots,p_n) \in \mathbb{N}^n$ is called
a {\em $\lambda$-parking function} of length $n$, if 
non-decreasing rearrangement
$p_{i_1} \le p_{i_2}
\le \ldots \le p_{i_n}$ of $\mathbf{p}$ satisfies 
$p_{i_j} < \lambda_{n-j+1}$ 
for $1 \le j \le n$. Let ${\rm PF}_n(\lambda)$ be the
set of $\lambda$-parking functions of length $n$. Then
$\mathbf{x}^{\mathbf{p}}$ is a standard monomial of 
$\frac{R}{I_{\mathfrak{S}_n}(\mathbf{u})^{[\mathbf{u_n+c-1}]}}$
if and only if $\mathbf{p} \in {\rm PF}_n(\lambda)$. Also, 
$\lambda$-parking functions for $\lambda=(n,n-1,\ldots,1)$ are
precisely (ordinary) parking functions of length $n$,
that is, ${\rm PF}_n((n,n-1,\ldots,1)) = {\rm PF}_n$.

The Alexander dual
$I_S^{[\mathbf{n}]}$ of $I_S$ is an order monomial ideal for  
$S=\mathfrak{S}_n(132,231)$,
$\mathfrak{S}_n(123,132)$ and $\mathfrak{S}_n(123,132,213)$
(see \cite{AC2, AC3}). The minimal generators of 
$I_W(\mathbf{u})^{[\mathbf{u_n+c-1}]}$ correspond to 
elements of poset $\Sigma_n$. The monomial ideal 
$I_W(\mathbf{u})^{[\mathbf{u_n+c-1}]}$ is also
an order monomial ideal and its minimal resolution is the 
cellular resolution supported on the order complex
$\Delta(\Sigma_n)$ of $\Sigma_n$. Thus,
the $i^{th}$ Betti number
$\beta_i(I_W(\mathbf{u})^{[\mathbf{u_n+c-1}]})
= (i!)~ S(n+1,i+1)$ for $0 \le i \le n-1$.

We now describe standard monomials of
$\frac{R}{I_W(\mathbf{u})^{[\mathbf{u_n+c-1}]}}$. Since
$I_W(\mathbf{u}) \subseteq I_{\mathfrak{S}_n}(\mathbf{u})$, 
we have $I_{\mathfrak{S}_n}(\mathbf{u})^{[\mathbf{u_n+c-1}]}
\subseteq I_W(\mathbf{u})^{[\mathbf{u_n+c-1}]}$. Hence,
standard monomials of
$\frac{R}{I_W(\mathbf{u})^{[\mathbf{u_n+c-1}]}}$ are of the
form $\mathbf{x}^{\mathbf{p}}$ for some $\mathbf{p}
\in {\rm PF}_n(\lambda)$.

\begin{definition}
{\rm A $\lambda$-parking function $\mathbf{p}=
(p_1,\ldots,p_n) \in {\rm PF}_n(\lambda)$ is 
said to be a {\em restricted $\lambda$-parking function} of
length $n$ if there exists a permutation $\alpha \in 
\mathfrak{S}_n$ such that $ p_{\alpha_i} < 
\mu_{\alpha_i,T_i}^{\mathbf{u}}$ for all $1 \le i \le n$, where
$\alpha_i=\alpha(i), ~T_1=[n],~ 
T_i=[n]\setminus \{\alpha_1,\ldots,\alpha_{i-1}\};~ (i \ge 2)$
and $\mu_{j,T}^{\mathbf{u}}$ is as 
in Proposition \ref{P1}}. 
\end{definition}  

Let $\widetilde{{\rm PF}}_n(\lambda)$
be the set of restricted $\lambda$-parking functions of
length $n$. For $\mathbf{u}=(1,2,\ldots,n)$ and $c=1$,
 we have $\lambda= (n,n-1,\ldots,1)$. In this case, 
a restricted $\lambda$-parking function is called a 
{\em  restricted parking function} of length $n$ and  we
simply  write $\widetilde{{\rm PF}}_n$ 
for $\widetilde{{\rm PF}}_n(\lambda)$.
Also, $\mu_{j,T} = \mu_{j,T}^{\mathbf{u}}$ is 
given by
  $\mu_{j_1,T} = n-t+1$ and $\mu_{j_i,T}= (n-t+1)-(j_i-i); 
  ~i \ge 2$, 
where $\emptyset \ne 
 T=\{j_1,\ldots,j_t\} \subseteq [n]$ with
  $j_1 < \ldots < j_t$.  

\begin{proposition}
A monomial $\mathbf{x}^{\mathbf{p}}$ is a standard monomial
of $\frac{R}{I_W(\mathbf{u})^{[\mathbf{u_n+c-1}]}}$
if and only if 
$\mathbf{p}\in \widetilde{{\rm PF}}_n(\lambda)$ is a 
restricted $\lambda$-parking function of length $n$, with
$\lambda_i = u_n-u_i+c; (1 \le i \le n)$. In particular,
a monomial $\mathbf{x}^{\mathbf{p}}$ is a standard monomial
of $\frac{R}{I_W^{[\mathbf{n}]}}$
if and only if 
$\mathbf{p}\in \widetilde{{\rm PF}}_n$ is a 
restricted parking function of length $n$.
\label{P2}
\end{proposition}
\begin{proof}
Standard monomials of 
$\frac{R}{I_W(\mathbf{u})^{[\mathbf{u_n}]}}$ are
characterized in Theorem 4.3 of \cite{AC1}. Proceeding on similar
lines, we get the desired result. \hfill $\square$
\end{proof}

Using the cellular resolution of 
$I_W(\mathbf{u})^{[\mathbf{u_n+c-1}]}$ supported on the order 
complex $\Delta(\Sigma_n)$, we obtain the multigraded
Hilbert series 
$H\left(\frac{R}{I_W(\mathbf{u})^{[\mathbf{u_n+c-1}]}}\right)$
of $\frac{R}{I_W(\mathbf{u})^{[\mathbf{u_n+c-1}]}}$. Proceeding 
as in the proof of Proposition 4.5 of \cite{AC1}, we get a combinatorial
formula 
\begin{eqnarray}
\label{EQ1}
| \widetilde{{\rm PF}}_n(\lambda) |
& = & 
\dim_k\left(\frac{R}{I_W(\mathbf{u})^{[\mathbf{u_n+c-1}]}}
\right) \\
& = & \sum_{i=1}^n (-1)^{n-i} \sum_{{\emptyset=A_0 \subsetneq A_1
\subsetneq \ldots \subsetneq A_i = [n]}} \prod_{q=1}^i 
\left( \prod_{j \in A_q \setminus A_{q-1}} 
\mu_{j,A_q}^{\mathbf{u}} 
\right) \nonumber
\end{eqnarray}
for enumeration of standard monomials
of $\frac{R}{I_W(\mathbf{u})^{[\mathbf{u_n+c-1}]}}$,
where $\mu_{j,A_q}^{\mathbf{u}}$ is as in Proposition \ref{P1}.
Let $\mathcal{C}$ be a chain in $\Sigma_n$ of the form
\[
\mathcal{C} : A_1 \subsetneq A_2 \subsetneq \ldots
\subsetneq A_i=[n] 
\]
of length $\ell(\mathcal{C})=i-1$ and 
 let $\mu^{\mathbf{u}}(\mathcal{C}) =
\prod_{q=1}^i 
\left( \prod_{j \in A_q \setminus A_{q-1}} 
\mu_{j,A_q}^{\mathbf{u}} \right) $, where $A_0=\emptyset$.
Suppose
$\mathfrak{Ch}([n])$ is the set of such chains $\mathcal{C}$
in $\Sigma_n$.
Then formula (\ref{EQ1}) can be expressed compactly as
\begin{equation}
| \widetilde{{\rm PF}}_n(\lambda) |
=  
\dim_k\left(\frac{R}{I_W(\mathbf{u})^{[\mathbf{u_n+c-1}]}}
\right)=
\sum_{\mathcal{C} \in \mathfrak{Ch}([n])} 
(-1)^{n - \ell(\mathcal{C}) - 1}  \mu^{\mathbf{u}}(\mathcal{C}) .
\label{EQ2}
\end{equation}

We now take $u_i=i$ in (\ref{EQ2}). For $ c \ge 1$,
let $\dim_k\left(\frac{R}{I_W^{[\mathbf{n+c-1}]}} \right)
 =a_n(c)$. Then we see that 
$a_n(c)$ is a polynomial expression in $c$ of degree $n$ for
$n \ge 1$. In fact, $a_1(c) = c$ and $a_2(c) = c^2+2c$. 

\begin{lemma}
Let $n \ge 3, \mathbf{u}=(1,2,\ldots,n)$ and $c \ge 1$. For
a chain $\mathcal{C} \in \mathfrak{Ch}[n]$ of length $i-1$ of 
the form $A_1 \subsetneq \ldots A_r 
\subsetneq A_{r+1} \subsetneq
\ldots \subsetneq A_i=[n]$ with $n \in
 A_{r+1}\setminus A_r$ and
$|A_{r+1} \setminus A_r| \ge 2$, there exists a unique chain,
namely $\widetilde{\mathcal{C}}: 
A_1 \subsetneq \ldots A_r \subsetneq A_r \cup 
\{n\} \subsetneq A_{r+1} \subsetneq
\ldots \subsetneq A_i=[n]$ in $\mathfrak{Ch}[n]$ of length
$i$ such that $\mu^{\mathbf{u}}(\mathcal{C})=
\mu^{\mathbf{u}}(\widetilde{\mathcal{C}})$.
\label{L1} 
\end{lemma}
\begin{proof}
Since $\mu^{\mathbf{u}}(\mathcal{C}) =
\prod_{q=1}^i 
\left( \prod_{j \in A_q \setminus A_{q-1}} 
\mu_{j,A_q}^{\mathbf{u}} \right) $, the equality
$\mu^{\mathbf{u}}(\mathcal{C})=
\mu^{\mathbf{u}}(\widetilde{\mathcal{C}})$ holds if
$\mu_{n,A_r \cup \{n\}}^{\mathbf{u}} = 
\mu_{n,A_{r+1}}^{\mathbf{u}}$. Clearly,
$\mu_{n,A_r \cup \{n\}}^{\mathbf{u}} = n-(|A_r|+1+n-(|A_r|+1))+c
=c$ and $\mu_{n,A_{r+1}}^{\mathbf{u}} = 
n-(|A_{r+1}|+n-|A_{r+1}|)+c=c$. \hfill $\square$
\end{proof}

Let ${\mathfrak{Ch}}'[n]$ be the set of chains in
$\Sigma_n$ obtained from $\mathfrak{Ch}[n]$ on deleting
chains $\mathcal{C}$ and $\widetilde{\mathcal{C}}$ appearing
in Lemma \ref{L1}. Then 
\[
a_n(c) = \sum_{\mathcal{C} \in \mathfrak{Ch}([n])} 
(-1)^{n - \ell(\mathcal{C}) - 1}  \mu^{\mathbf{u}}(\mathcal{C})
=\sum_{\mathcal{C} \in \mathfrak{Ch}'([n])} 
(-1)^{n - \ell(\mathcal{C}) - 1}  \mu^{\mathbf{u}}(\mathcal{C}) . 
\]
For $\mathbf{u}=(1,2,\ldots,n)$ and $c \ge 1$, the value 
$\mu^{\mathbf{u}}(\mathcal{C})$
depends on the chain $\mathcal{C}$ and $c$. Thus,
 we write $\mu^{c}(\mathcal{C})$
for $\mu^{\mathbf{u}}(\mathcal{C})$. Hence,
$a_n(c) = \sum_{\mathcal{C} \in \mathfrak{Ch}([n])} 
(-1)^{n - \ell(\mathcal{C}) - 1}  \mu^{c}(\mathcal{C})
=\sum_{\mathcal{C} \in \mathfrak{Ch}'([n])} 
(-1)^{n - \ell(\mathcal{C}) - 1}  \mu^{c}(\mathcal{C})$. 

For $n \ge 3$, the chains in ${\mathfrak{Ch}}'[n]$ can
 be divided
into three types. 
\begin{itemize}
\item A chain $\mathcal{C}: A_1 \subsetneq \ldots
\subsetneq A_i=[n]$ in ${\mathfrak{Ch}}'[n]$ is called a 
{\em Type-I} chain if $A_1=\{n\}$. The Type-I chains in
${\mathfrak{Ch}}'[n]$ are in one-to-one correspondence with
chains in ${\mathfrak{Ch}}[n-1]$. This correspondence is
given by 
\[\mathcal{C} \mapsto \mathcal{C}\setminus A_1:
A_2\setminus \{n\}
\subsetneq \ldots \subsetneq A_i \setminus \{n\} =[n-1].
\]
As $\ell(\mathcal{C}) -1=\ell(\mathcal{C}\setminus A_1)$ and
$\mu^c(\mathcal{C}) = (n-1+c)~ \mu^c(\mathcal{C}\setminus A_1)$,
we have
\[
\sum_{\substack{\mathcal{C}\in \mathfrak{Ch}'[n]; 
\\ {\rm Type-I}}} 
(-1)^{n-\ell(\mathcal{C})-1} ~ \mu^c(\mathcal{C}) = 
(n-1+c)~ a_{n-1}(c).
\]
\item A chain $\mathcal{C}: A_1 \subsetneq \ldots
\subsetneq A_i=[n]$ in ${\mathfrak{Ch}}'[n]$ is called a 
{\em Type-II} chain if $A_{i-1}=[n-1]$. The Type-II chains in 
${\mathfrak{Ch}}'[n]$ are in one-to-one correspondence with
chains in ${\mathfrak{Ch}}[n-1]$. This correspondence is
given by 
\[\mathcal{C} \mapsto \mathcal{C}|_{[n-1]}:A_1
\subsetneq \ldots \subsetneq A_{i-1} =[n-1].
\]
As $\ell(\mathcal{C}) -1=\ell(\mathcal{C}|_{[n-1]})$ and
$\mu^c(\mathcal{C}) = (c) ~\mu^{c+1}(\mathcal{C}|_{[n-1]})$,
we have
\[
\sum_{\substack{\mathcal{C}\in \mathfrak{Ch}'[n]; 
\\ {\rm Type-II}}} 
(-1)^{n-\ell(\mathcal{C})-1} ~\mu^c(\mathcal{C}) = 
(c)~ a_{n-1}(c+1).
\]
\item A chain $\mathcal{C}: A_1 \subsetneq \ldots
\subsetneq A_i=[n]$ in ${\mathfrak{Ch}}'[n]$ is called a 
{\em Type-III} chain if $n \in A_1$ and $|A_1| \ge 2$. The 
Type-III chains in ${\mathfrak{Ch}}'[n]$ are in one-to-one correspondence with
chains in ${\mathfrak{Ch}}[n-1]$. This correspondence is
given by 
\[\mathcal{C} \mapsto \mathcal{C}\setminus \{n\}:A_1\setminus \{n\}
\subsetneq \ldots \subsetneq A_i \setminus \{n\} =[n-1].
\]
As $\ell(\mathcal{C})= \ell(\mathcal{C}\setminus \{n\})$ and
$\mu^c(\mathcal{C}) = (c)~ \mu^{c}(\mathcal{C}\setminus \{n\})$,
we have
\[
\sum_{\substack{\mathcal{C}\in \mathfrak{Ch}'[n]; 
\\ {\rm Type-III}}} 
(-1)^{n-\ell(\mathcal{C})-1}~ \mu^c(\mathcal{C}) = 
(-c)~ a_{n-1}(c).
\]
\end{itemize}
Consider 
the poset $\Sigma_n$ and form
a poset $\Lambda_n=\Sigma_{n-1} \coprod 
(\Sigma_{n-1} \ast \{n\})$; for $n \ge 2$, where
$\Sigma_{n-1} \ast \{n\} = 
\{A \cup \{n\} : A \in \Sigma_{n-1}\}$ is a subposet of
$\Sigma_n$. Two elements
$A,B \in \Lambda_n$ are comparable if either
$A,B \in \Sigma_{n-1}$ are comparable or $A,B \in \Sigma_{n-1} \ast \{n\}$ are comparable
or $ \{A,B\}=\{[n-1],[n]\}$. The Hasse diagram 
of $\Lambda_n$ for
$n=3, 4$ are given in {\sc Figure-1}.

\begin{figure}[!hbpt]
{\begin{tikzpicture}
\node (max) at (0,1.5) {$123$};
\node (a) at (-1.5,0) {$12$};
\node (b) at (0,0) {$13$};
\node (c) at (1.5,0) {$23$};
\node (d) at (-1.5,-1.5) {$1$};
\node (e) at (0,-1.5) {$2$};
\node (f) at (0,-2.5) {$\Lambda_3$};
\draw (d) -- (a) -- (e);
\draw (a) -- (max) -- (b)(max) -- (c);
\end{tikzpicture}}
\hspace{1.5em}
{\begin{tikzpicture}
\node (max) at (10.5,3) {$1234$};
\node (a) at (5.5,1.5) {123};
\node (b) at (9,1.5) {124};
\node (c) at (10.5,1.5) {134};
\node (d) at (12,1.5) {234};
\node (e) at (4,0) {12};
\node (f) at (5.5,0) {13};
\node (g) at (9,0) {14};
\node (h) at (7,0) {23};
\node (i) at (10.5,0) {24};
\node (j) at (12,0) {34};
\node (k) at (4,-1.5) {1};
\node (l) at (5.5,-1.5) {2};
\node (m) at (7,-1.5) {3};
\node (n) at (9,-2.5) {$\Lambda_4$};
\draw (k) -- (e) -- (a);
\draw (a) -- (max) -- (b) -- (g) -- (c) -- (max) -- (d) -- (i)(c) -- (j) -- (d)(i) -- (b);
\draw(k) -- (f) -- (a) -- (h) -- (l) -- (h) -- (m)(l) -- (e)(m) -- (f);
\end{tikzpicture}}
\caption{}
\end{figure}

Clearly, Type-II chains in $\mathfrak{Ch}'[n]$ are chains in
$\Lambda_{n}$ with an edge $[n-1]\subsetneq [n]$, 
while Type-III chains in
$\mathfrak{Ch}'[n]$ are chains in $\Lambda_n$ containing $[n]$ 
but not $[n-1]$.
\begin{proposition}
For $n \ge 3$ and $c \ge 1$, $a_n(c) = 
\dim_k\left(\frac{R}{I_W^{[\mathbf{n+c-1}]}}\right)$ satisfies 
the recurrence relation
\[
a_n(c) = (n-1)a_{n-1}(c) + c~ a_{n-1}(c+1).
\]
\label{P3}
\end{proposition}
\begin{proof}
As $a_n(c) = \sum_{\mathcal{C} \in \mathfrak{Ch}([n])} 
(-1)^{n - \ell(\mathcal{C}) - 1}  \mu^{c}(\mathcal{C})
=\sum_{\mathcal{C} \in \mathfrak{Ch}'([n])} 
(-1)^{n - \ell(\mathcal{C}) - 1}  \mu^{c}(\mathcal{C})$, we have
\begin{eqnarray*}
a_n(c) & = & \left[\sum_{\substack{\mathcal{C}\in \mathfrak{Ch}'[n]; 
\\ {\rm Type-I}}} + \sum_{\substack{\mathcal{C}\in \mathfrak{Ch}'[n]; 
\\ {\rm Type-II}}} + \sum_{\substack{\mathcal{C}\in \mathfrak{Ch}'[n]; 
\\ {\rm Type-III}}} \right]
~(-1)^{n-\ell(\mathcal{C})-1}~ \mu^c(\mathcal{C}) \\
& = & (n-1+c)~ a_{n-1}(c) + (c)~ a_{n-1}(c+1) +(-c)~a_{n-1}(c)\\
& = & (n-1)~a_{n-1}(c) + (c)~a_{n-1}(c+1).
\end{eqnarray*}
\hfill $\square$
\end{proof}
Replacing $c$ by an indeterminate $x$, we consider polynomial
$a_n(x)$. The recurrence relation in Proposition \ref{P3} holds
for all $ c \ge 1$, thus there exists a polynomial identity
\begin{equation}
a_n(x) = (n-1)~ a_{n-1}(x) + x~a_{n-1}(x+1) 
\quad {\rm for}~~n \ge 3.
\label{EQ3}
\end{equation} 
Since $a_1(x) = x$ and $a_2(x) = x^2+2x$, on setting $a_0(x) =1$,
the recurrence relation (\ref{EQ3}) is valid for $n \ge 1$. Note
that $a_n(0) = 0$ for $n \ge 1$.

\begin{proposition}
For $n \ge 1$, $a_n(x) = 
\sum_{r=1}^n s(n,r)~ x(x+1)\cdots(x+r-1)$. 
\label{P4}
\end{proposition}
\begin{proof}
Let $x^{\bar{r}} = x(x+1)\cdots(x+r-1)$ be the $r^{th}$ 
{\em rising power} of $x$. Then 
$\left\{ x^{\bar{r}} : r=0,1,\ldots\right\}$
is a $\mathbb{Q}$-basis of $\mathbb{Q}[x]$, where $x^{\bar{0}}=1$.
As $a_n(0)=0$ for $n \ge 1$, we can express
$a_n(x) = \sum_{r=1}^n \alpha_n(r) x^{\bar{r}} $. As $a_n(x)$
satisfy recurrence relation (\ref{EQ3}) for $n\ge 1$, it follows
that $\alpha_n(r)$ and the (signless) Stirling number $s(n,r)$ 
of the first kind satisfy the same 
recurrence relation with the same initial conditions (see \cite{St}).
Thus $\alpha_n(r) = s(n,r)$. \hfill $\square$
\end{proof}

\begin{theorem}
For $ n \ge 1$, $\dim_k\left(\frac{R}{I_W^{[\mathbf{n}]}}\right)
=a_n =
\sum_{r=1}^n (r!)~ s(n,r)$.
\label{T1}
\end{theorem}
\begin{proof}
Since $a_n = a_n(1)$, theorem follows from Proposition
\ref{P4}. \hfill $\square$
\end{proof}

Consider the integer sequence (A007840) in OEIS \cite{S}. The
$n$th term $b_n$ of this sequence is the number of factorization
of permutations of $[n]$ into ordered cycles and 
$b_n = \sum_{r=1}^n (r!) ~s(n,r)$. It can be verified that
\[
b_n = {\rm Per}([m_{ij}]_{n \times n}) =
{\rm Per} \begin{bmatrix}
1&1&1&\ldots& 1\\1&2&1&\ldots & 1\\1&1&3&\ldots &1 \\
\vdots& \vdots & \vdots & \ddots & \vdots\\
1&1&1& \ldots & n
\end{bmatrix} ,
\]
where $m_{ii}=i$ and $m_{ij}=1$ for $i \ne j$. We recall that 
{\em permanent} ${\rm Per}([m_{ij}]_{n \times n})$ 
of the matrix $[m_{ij}]_{n \times n}$
is given by
$\sum_{\sigma \in \mathfrak{S}_n} \prod_{i=1}^n m_{i\sigma(i)}$.
There are many combinatorial interpretation of the integer 
sequence (A007840). Theorem \ref{T1} gives a description of
the integer sequence (A007840) in terms of enumeration
of standard monomials of $\frac{R}{I_W^{[\mathbf{n}]}}$, or
equivalently, in terms of the number $|\widetilde{{\rm PF}}_n|$
of restricted parking functions of length $n$.

We now show that enumeration of standard monomials of 
$\frac{R}{I_W^{[\mathbf{n}]}}$ is related to enumeration of
rooted-labelled unimodal forests on $[n]$. The concept of 
permutations avoiding patterns has been extended to 
many combinatorial  objects, such as, trees, graphs and 
posets. Let $F_n$ be the set of (unordered) rooted-labelled
forests on the vertex set $[n]$. Then $|F_n|=(n+1)^{n-1}$. A 
rooted-labelled forest on $[n]$ is said to {\em avoid a
pattern} $\tau \in \mathfrak{S}_r$ if along each path from
a root to a vertex, the sequence of labels do not contain a 
subsequence with the same relative order as in the
patterns $\tau=\tau(1) \tau(2) \ldots \tau(r)$. Let $F_n(\tau)$
be the set of rooted-labelled forests on $[n]$ that avoid 
pattern $\tau$. For example, if $\tau = 21$ is a transposition,
then $F_n(21)$ is the set of rooted-labelled increasing 
forests on $[n]$. In other words, labels on any path from a root
to a vertex for a forest in $F_n(21)$ form an increasing 
sequence. Let $F_n(\tau^{(1)},\ldots,\tau^{(s)})$ be the set of
rooted-labelled forests on $[n]$ that avoid a set 
$\{\tau^{(1)},\ldots,\tau^{(s)}\}$ of patterns. The enumeration
of rooted-labelled forests on $[n]$ that avoid
various patterns are obtained in \cite{AnAr}. In particular,
it is shown  that 
$|F_n(213,312)|=\sum_{r=1}^n (r!) ~s(n,r)$ for
$n \ge 1$. The rooted-labelled forests on $[n]$ avoiding 
$213$ and $312$-patterns are precisely 
the {\em unimodal} forests.
Since $|\widetilde{{\rm PF}}_n|=|F_n(213,312)|$, an 
explicit or algorithmic bijection 
$\phi : \widetilde{{\rm PF}}_n
\longrightarrow F_n(213,312)$ is desired.

Before we end this section, we describe an easy extension of
Theorem \ref{T1}.

Let $b,c \ge 1$ and 
$\mathbf{u}=(u_1,\ldots,u_n) \in \mathbb{N}^n$
with $u_i = u_1+(i-1)b$. We have seen that the standard
monomials of $\frac{R}
{I_{\mathfrak{S}_n}(\mathbf{u})^{[\mathbf{u_n+c-1}]}}$ are
of the form $\mathbf{x}^{\mathbf{p}}$, where $\mathbf{p}
\in {\rm PF}_n(\lambda)$ is a $\lambda$-parking function
of length $n$ and $\lambda_i =u_n -u_i +c = (n-i)b+c$. Then
$|{\rm PF}_n(\lambda)|= c (c+nb)^{n-1}$ (see \cite{PiSt, PoSh}).
Let $|\widetilde{{\rm PF}}_n(\lambda)|= \dim_k\left(
\frac{R}{I_{W}(\mathbf{u})^{[\mathbf{u_n+c-1}]}}\right) =
\widetilde{a_n}(c)$. 
Actually, $\widetilde{a_n}(c)$ depends on $b$ also, but we are 
treating $b$ to be a fixed constant. Also, 
$\widetilde{a_n}(c)$ is a polynomial expression in $c$.
\begin{proposition}
For $n \ge 3, ~b,c \ge 1$, $\widetilde{a_n}(c)$ satisfies
a recurrence relation 
\[
\widetilde{a_n}(c) = ((n-1)b)~\widetilde{a_{n-1}}(c) +
(c) ~\widetilde{a_{n-1}}(c+b).
\]
\label{P5}
\end{proposition}
\begin{proof}
From equation (\ref{EQ2}), we have
\[ \widetilde{a_n}(c) =
\dim_k\left(\frac{R}{I_W(\mathbf{u})^{[\mathbf{u_n+c-1}]}}
\right)=
\sum_{\mathcal{C} \in \mathfrak{Ch}([n])} 
(-1)^{n - \ell(\mathcal{C}) - 1} ~ \mu^{\mathbf{u}}(\mathcal{C}), 
\]
where $u_i = u_1+(i-1)b$. For such $\mathbf{u}$, Lemma \ref{L1} holds. Thus 
\[
\widetilde{a_n}(c) = \sum_{\mathcal{C} \in \mathfrak{Ch}([n])} 
(-1)^{n - \ell(\mathcal{C}) - 1}  ~\mu^{\mathbf{u}}(\mathcal{C})
=\sum_{\mathcal{C} \in \mathfrak{Ch}'([n])} 
(-1)^{n - \ell(\mathcal{C}) - 1}  ~\mu^{\mathbf{u}}(\mathcal{C}) . 
\]
Now proceed as in the proof of Proposition \ref{P3}.
 \hfill $\square$
\end{proof}
Replacing $c$ with an indeterminate $x$, we consider polynomial
$\widetilde{a_n}(x)$. Thus there is a polynomial identity
\begin{equation}
\widetilde{a_n}(x) = ((n-1)b)~ \widetilde{a_{n-1}}(x) 
+ x~ \widetilde{a_{n-1}}(x+b) 
\quad {\rm for}~~n \ge 3.
\label{EQ4}
\end{equation} 
Since $\widetilde{a_1}(x) = x$ and 
$\widetilde{a_2}(x) = x^2+2bx$, on setting 
$\widetilde{a_0}(x) =1$,
the recurrence relation (\ref{EQ4}) is valid for $n \ge 1$. Again, 
we have  $\widetilde{a_n}(0) = 0$ for $n \ge 1$.
\begin{theorem}
For $n \ge 1$, $\widetilde{a_n}(x) = \sum_{r=1}^n
(b^{n-r} ~s(n,r))~x(x+b)\cdots(x+(r-1)b)$. In particular, for
$\lambda=(\lambda_1, \ldots, \lambda_n)$ with
$\lambda_i=(n-i)b+c$
\[
|\widetilde{{\rm PF}}_n(\lambda)|
= \widetilde{a_n}(c) = b^n ~\sum_{r=1}^n s(n,r)~
\frac{\Gamma(\frac{c}{b}+r)}{\Gamma(\frac{c}{b})},
\]
where $\Gamma$ is the gamma function, i.e., $\Gamma(x+1)=
x ~\Gamma(x)$ for $x > 0$ and $\Gamma(1)=1$.
\label{T2}
\end{theorem}
\begin{proof}
As in the proof of Theorem \ref{T1}, let
\[\widetilde{a_n}(x) = \sum_{r=1}^n \widetilde{\alpha_n}(r)~
x(x+b)\cdots(x+(r-1)b).
\]
Then from recurrence relation (\ref{EQ4}), 
$\widetilde{\alpha_n}(r)$ satisfies the recurrence relation
\[\widetilde{\alpha_n}(r) = (n-1)b~\widetilde{\alpha_{n-1}}(r)
+ \widetilde{\alpha_{n-1}}(r-1); \quad {\rm for}~1 \le r\le n,
\]
with initial conditions $\widetilde{\alpha_0}(1)=0$ and
$\widetilde{\alpha_1}(1)=1$. It is straight forward to see that
$\widetilde{\alpha_n}(r) = b^{n-r}~s(n,r)$. \hfill $\square$
\end{proof}

\section{Some other cases}
The Betti numbers and enumeration of standard monomials of
the Artinian  quotient $\frac{R}{I_S^{[\mathbf{n}]}}$ for
$ S = \mathfrak{S}_n(132,231), \mathfrak{S}_n(123,132)$ and
$\mathfrak{S}_n(123,132,213)$ are give in \cite{AC2,AC3}. In this
section, the monomial ideal $I_S$ and its Alexander dual
$I_S^{[\mathbf{n}]}$ are studied for various other 
subsets $S \subseteq \mathfrak{S}_n$ consisting of permutations
avoiding patterns. For clarity of presentation, we divide these
subsets into three cases.
\begin{enumerate}
\item[Case 1.] \quad $S_1 = \mathfrak{S}_n(123,132,312),
S_2= \mathfrak{S}_n(123,213,231),
S_3 = \mathfrak{S}_n(132,213,231)$.
\item[Case 2.] \quad $T_1 = \mathfrak{S}_n(123,132,231),
T_2 = \mathfrak{S}_n(213,312,321)$.
\item[Case 3.] \quad $U = \mathfrak{S}_n(123,231,312)$.
\end{enumerate}

We have, $|S_a|=|T_b|=|U|=n$ for $1\le a \le 3$ and 
$1\le b \le 2$ (see \cite{SiSc}).
\begin{lemma}
The minimal generators of the Alexander dual
$I_S^{[\mathbf{n}]}$ for $S = S_a,T_b$ or $U$ are given
as follows.
\begin{enumerate}
\item[{\rm (i)}] $I_{S_1}^{[\mathbf{n}]} = \left\langle 
x_{\ell}^{\ell+1},~ x_i^i \left( \prod_{j > i}^n x_j \right)
 : 1 \le \ell \le n-1;~1 \le i \le n \right\rangle$.
\item[{\rm (ii)}] $I_{S_2}^{[\mathbf{n}]} = \left\langle 
x_{\ell}^{n}, ~x_i^i  x_j^{j-1} 
 : 1 \le \ell \le n;~1 \le i < j \le n \right\rangle$.
\item[{\rm (iii)}] $I_{S_3}^{[\mathbf{n}]} = \left\langle 
x_{\ell}^{n}, ~x_i^i  x_j^{n-(j-i)} 
 : 1 \le \ell \le n;~1 \le i < j \le n \right\rangle$.
\item[{\rm (iv)}] $I_{T_1}^{[\mathbf{n}]} = \left\langle 
x_{\ell}^{\ell+1},~ x_n^n, ~ x_i^i  x_n^i 
 : 1 \le \ell \le n-1;~1 \le i < n \right\rangle$.
\item[{\rm (v)}] $I_{T_2}^{[\mathbf{n}]} = \left\langle 
x_{\ell}^{n-\ell+1},~ x_n^n, ~x_i^{n-i}  x_n^{n-i} 
 : 1 \le \ell \le n-1;~1 \le i < n \right\rangle$.
\item[{\rm (vi)}] $I_{U}^{[\mathbf{n}]} = \left\langle 
 \prod_{j \in A } x_j^{\nu_{j,A}} 
 : A =\{j_1,\ldots,j_t\} \in \Sigma_n \right\rangle$, where
 $\nu_{j_1,A} = n -(j_{|A|}-j_1)$ and $\nu_{j_i,A}= j_i - j_{i-1}$
 for $i \ge 2$, provided $j_1 < j_2 <\ldots < j_t$.
\end{enumerate}
\label{L2}
\end{lemma}
\begin{proof}
We recall that a vector $\mathbf{b}\in \mathbb{N}^n$
satisfying $\mathbf{b} \le \mathbf{n}$ (i.e., $b_i \le n$)
is maximal with 
$\mathbf{x}^{\mathbf{b}} \notin I_S$ if and only if 
$\mathbf{x}^{\mathbf{n}-\mathbf{b}}$ is a minimal generator
of $I_S^{[\mathbf{n}]}$ (see Proposition 5.23 of \cite{MS}).
Now proceeding as in the proof of Lemma 2.1 and 2.2 of
\cite{AC3}, it is easy to get the minimal generators of
the Alexander duals. We sketch a proof of part (i) and (vi) as 
proof of other parts are on similar lines.
 
For $\ell \in [n-1]$, let $\mathbf{b}_{\ell}=(n,\ldots,
n-{\ell}-1,\ldots,n)$ ($\ell^{th}$ coordinate $n-\ell -1$, elsewhere $n$).
Then $\mathbf{x}^{\mathbf{b}_{\ell}} \notin I_{S_1}$ and this
gives the minimal generator $x_{\ell}^{\ell+1} 
\in I_{S_1}^{[\mathbf{n}]}$. For $i \in [n]$, let
$\mathbf{b}_{i,n}=(n,\ldots,n,n-i,n-1,\ldots,n-1) \in 
\mathbb{N}^n$ (i.e., $i^{th}$ coordinate $n-i$, first $i-1$ coordinates $n$, 
and the last $n-i$ coordinates $n-1$). Again, 
$\mathbf{x}^{\mathbf{b}_{i,n}} \notin I_{S_1}$ and this gives
the minimal generator $x_i^i (x_{i+1}\ldots x_n) \in 
I_{S_1}^{[\mathbf{n}]}$. This proves part (i).

If $A=\{\ell\} \in \Sigma_n$, then taking 
$\widehat{\mathbf{b}}_{\ell}
= (n,\ldots,0,\ldots,n)$ (i.e., $0$ at $\ell^{th}$ place 
and elsewhere $n$), we get the minimal generator $x_{\ell}^n
\in I_U^{[\mathbf{n}]}$. For $A =\{j_1,\ldots, j_t\} \in 
\Sigma_n$ with $t \ge 2$ and $j_1 < \ldots < j_t$, let 
$\widehat{\mathbf{b}}_A = (b_1,\ldots,b_n)$, where
$b_{j_1}=j_t - j_1, ~b_{j_i}=n-(j_i-j_{i-1})$ 
(for $ i \ge 2$) and
$b_r = n $ (for $r \notin A$).\\
Claim : \quad $\mathbf{x}^{\widehat{\mathbf{b}}_A}
\notin I_U$.

Otherwise, there exists a $\sigma \in U$ such that 
$\mathbf{x}^{\sigma}$ divides 
$\mathbf{x}^{\widehat{\mathbf{b}}_A}$. Thus
$\sigma(j_1) \le j_t- j_1$ and $\sigma(j_i) 
\le n-(j_i - j_{i-1})$ for $2 \le i \le t$. We see that 
\[ \sigma(j_1) > \sigma(j_2) > \ldots > \sigma(j_t).
\]
If $\sigma(j_{i-1}) < \sigma(j_i)$ for $1 < i \le t$, then 
$ \sigma(j_{i-1}), \sigma(j_i) \in [n-(j_i-j_{i-1})]$. But
$|[n-(j_i-j_{i-1})]|= n-(j_i-j_{i-1})$ and 
$|[j_{i-1}] \coprod [j_i,n]| = n-(j_i-j_{i-1}) +1$, where
$[a,b] = \{m \in \mathbb{Z} : a \le m \le b\}$ denotes an integer
interval for $a,b \in \mathbb{Z}$. Thus there exists
$\ell \in [n] \setminus [j_{i-1},j_i]$ such that
$\sigma(\ell) \notin [n-(j_i-j_{i-1})]$. This shows that
$\sigma(j_{i-1}) < \sigma(j_i) < \sigma(\ell)$. Hence,
$\sigma$ has a 123 or a 312-pattern, a contradiction to 
$\sigma \in U$. Now $\sigma(j_t) < \sigma(j_1) \le j_t-j_1$
implies that $j_t -j_1 \ge 2$. Again, 
$\sigma(j_1), \sigma(j_t) \in [j_t-j_1]$, but
$|[j_t-j_1]| = j_t-j_1 < |[j_1,j_t]|=j_t-j_1+1$. Thus there
exists $\ell \in [j_1+1,j_t-1]$ such that 
$\sigma(\ell) > j_t -j_1$. This shows that, 
$\sigma(j_t) < \sigma(j_1)
< \sigma(\ell)$ with $j_1 < \ell < j_t$ demonstrating that 
$\sigma$ has a 231-pattern, a contradiction. This proves our 
claim. It can be shown that $\widehat{\mathbf{b}}_A$ has the
desired maximality property and hence $\mathbf{x}^{\mathbf{n}
- \widehat{\mathbf{b}}_A}$ is a minimal generator of 
$I_U^{[\mathbf{n}]}$. \hfill $\square$
\end{proof}

We shall show that all monomial ideals in Lemma \ref{L2} are
order monomial ideals. Let $(P, \preceq)$ be a finite poset
and let $\{\omega_u : u \in P\}$ be a set of monomials
in $R$. The monomial ideal $I = \langle \omega_u : 
u \in P \rangle$ is said to be an 
{\em order monomial ideal}
if for any pair $u,v \in P$, there is
an upper bound $w \in P$ of $u$ and $v$ such that
$\omega_w$ divides the least common multiple 
${\rm LCM}(\omega_u,\omega_v)$ of $\omega_u$ and $\omega_v$. The
order complex
$\Delta(P)$ of a  finite poset $P$ is a simplicial complex, whose
$r$-dimensional faces are chains $u_1 \prec u_2 \prec
\ldots \prec u_{r+1}$ of length $r$ in $P$. If $F$ is a face
of $\Delta(P)$, then monomial label 
$\mathbf{x}^{\alpha(F)}$ (say) on $F$ is the 
${\rm LCM}(\omega_u : u \in F)$. Let 
\[\mathbb{F}_{*}(\Delta(P)) : \cdots \rightarrow \mathbb{F}_i \rightarrow
\mathbb{F}_{i-1} \rightarrow \cdots \rightarrow \mathbb{F}_1
\rightarrow \mathbb{F}_0 \rightarrow 0
\]
be the free $R$-complex associated to the (labelled) simplicial
complex $\Delta(P)$. If $\mathbb{F}_{*}(\Delta(P))$ is 
exact at $\mathbb{F}_i$ for $i \ge 1$, then we say that 
$\mathbb{F}_{*}(\Delta(P))$ is a 
{\em cellular resolution} of $I$
supported on $\Delta(P)$ (see \cite{BPS, BS, MS}).

It is convenient to study the monomial ideal $I_S$ in Lemma \ref{L2} 
according to the three cases already described. \\
{\sc Case-1.} To each monomial ideal $I_{S_a}^{[\mathbf{n}]}$, we 
associate a poset $\Sigma_n(S_a)$ (for $1 \le a \le 3$) as follows.
\begin{enumerate}
\item[(i)] Let $\Sigma_n(S_1) = \{\{\ell\}: 1 \le \ell \le n-1\}
\cup \{[i,n]: 1 \le i \le n \}$, where
$[i,n]=\{a \in \mathbb{N}: i \le a \le n\}$ and $[n,n]=\{n\}$.
We define a poset structure on $\Sigma_n(S_1)$ by describing
cover relations. For $\ell, \ell' \in [n-1]$ and $i,i' \in [n]$,
$\{\ell\}$ covers $\{\ell'\}$ (or $[i',n]$), if $\ell'=\ell+1$
(respectively, $i'=\ell+2$). Also, $[i,n]$ covers $\{\ell'\}$
(or $[i',n]$) if $i=\ell'$ (respectively, $i'=i+1$). 
The monomial labels $\omega_{\{\ell\}}=x_{\ell}^{\ell+1}$ and 
$\omega_{[i,n]}= x_i^i x_{i+1}\ldots x_n$. Set $\mu_{j,C}^{1}$
 for $C \in \Sigma_n(S_1)$ so that 
 $\omega_C = \prod_{j \in C} x_j^{\mu_{j,C}^{1}}$.
The finite
poset $\Sigma_n(S_1)$ appeared in \cite{AC3}.
\item[(ii)] Let $\Sigma_n(S_2) = \{\{\ell\}: 1 \le \ell \le n\}
\cup \{\{i,j\}: 1 \le i < j \le n \}$.
A poset structure on $\Sigma_n(S_2)$ is given by the following 
cover relations. For $i,j,i',j' \in [n]$ with $i < j$ and 
$i' < j'$,
$\{i,j\}$ covers $\{i',j'\}$, if either
($i=i'$ and $j'=j+1$) or ($j=i'$ and $j'=j+1$). Also,
$\{i,j\}$ covers $\{i'\}$
if either ($i=i'$ and $j=n$) or ($i'=j=n$).
In this case, the monomial labels $\omega_{\{\ell\}}=x_{\ell}^{n}$ and 
$\omega_{\{i,j\}}= x_i^i x_{j}^{j-1}$. Set $\mu_{j,C}^{2}$
 for $C \in \Sigma_n(S_2)$ so that 
 $\omega_C = \prod_{j \in C} x_j^{\mu_{j,C}^{2}}$.
\item[(iii)] Let $\Sigma_n(S_3) = \{\{\ell\}: 1 \le \ell \le n\}
\cup \{\{i,j\}: 1 \le i < j \le n \}$.
Again, a poset structure on $\Sigma_n(S_3)$ is given by the following 
cover relations. For $i,j,i',j' \in [n]$ with $i < j$ and 
$i' < j'$,
$\{i,j\}$ covers $\{i',j'\}$, if either
($i=i'$ and $j=j'+1$) or ($i=i'-1$ and $j'=j$). Also,
$\{i,j\}$ covers $\{i'\}$
if either ($i=i'$ and $j=i+1$) or ($j=i'$ and $j=i+1$).
Again, the monomial labels 
$\omega_{\{\ell\}}=x_{\ell}^{n}$ and 
$\omega_{\{i,j\}}= x_i^i x_{j}^{n-(j-i)}$. 
Set $\mu_{j,C}^{3}$
 for $C \in \Sigma_n(S_3)$ so that 
 $\omega_C = \prod_{j \in C} x_j^{\mu_{j,C}^{3}}$.

\end{enumerate}

The Hasse diagrams of 
$\Sigma_4(S_1), ~ \Sigma_4(S_2),~\Sigma_4(S_3)$ are given in {\sc Figure-2}.

\begin{figure}[!hbpt]
{\begin{tikzpicture}
\node (max) at (-1,2) {$1234$};
\node (a) at (-2,1) {$1$};
\node (c) at (0,1) {$234$};
\node (d) at (-2,0) {$2$};
\node (g) at (-2,-1) {$3$};
\node (b) at (0,-1) {$4$}; 
\node (f) at (0,0) {$34$};
\node (h) at (-1,-2) {$\Sigma_4(S_1)$};
\draw (d) -- (a) -- (max)(c) -- (max)(f) -- (c)(d) -- (g)(f) -- (b) -- (d) -- (c)(g) -- (f) -- (a);
\end{tikzpicture}}
\hspace{2.1em}
{\begin{tikzpicture}
\node (max) at (6,2) {$12$};
\node (a) at (5,1) {$13$};
\node (b) at (4,0) {$14$};
\node (c) at (3,-1) {$1$};
\node (d) at (7,1) {$23$};
\node (e) at (8,0) {$34$};
\node (f) at (9,-1) {$4$};
\node (g) at (6,0) {$24$};
\node (h) at (5,-1) {$2$};
\node (i) at (7,-1) {$3$};
\node (m) at (6,-2) {$\Sigma_4(S_2)$};
\draw (c) -- (b) -- (a) -- (max) -- (d) -- (e) -- (f) -- (g) -- (d)(h) -- (g)(i) -- (e)(f) -- (b)(e) -- (a);
\end{tikzpicture}}
\hspace{1.1em}
{\begin{tikzpicture}
\node (max) at (14,2) {$14$};
\node (a) at (13,1) {$13$};
\node (b) at (12,0) {$12$};
\node (c) at (11,-1) {$1$};
\node (d) at (15,1) {$24$};
\node (e) at (16,0) {$34$};
\node (f) at (17,-1) {$4$};
\node (g) at (14,0) {$23$};
\node (h) at (13,-1) {$2$};
\node (i) at (15,-1) {$3$};
\node (m) at (14,-2) {$\Sigma_4(S_3)$};
\draw (c) -- (b) -- (a) -- (max) -- (d) -- (e) -- (f)(i) -- (e)(i) -- (g) -- (a)(h) -- (g)(h) -- (b)(g) -- (d);
\end{tikzpicture}}
\caption{}
\end{figure}

\begin{proposition}
{\rm (i)}. ~The ideal $I_{S_a}^{[\mathbf{n}]}$ is an order
monomial ideal for $1 \le a \le 3$. \\
{\rm (ii)}.~The free complex 
$\mathbb{F}_{*}(\Delta(\Sigma_n(S_a)))$ is the cellular 
resolution of $I_{S_a}^{[\mathbf{n}]}$ supported on the 
order complex $\Delta(\Sigma_n(S_a))$ for $1 \le a \le 3$.
\label{P6}
\end{proposition} 
\begin{proof}
Given the poset structure on $\Sigma_n(S_a)$, it is a
straight forward verification that $I_{S_a}^{[\mathbf{n}]}$
is an order monomial ideal. Postnikov and Shapiro \cite{PoSh}
 showed that the free complex
 $\mathbb{F}_{*}(\Delta(P))$ is a
 cellular resolution of the order monomial ideal
 $I = \langle \omega_u : u \in P \rangle$ (see Theorem 2.4 of
 \cite{AC3}). \hfill $\square$
\end{proof}
\begin{remark}
{\rm The cellular resolution 
$\mathbb{F}_{*}(\Delta(\Sigma_n(S_a)))$ is minimal for $a=1$,
but nonminimal for $a=2,3$. Also, the $r^{th}$ Betti number
$\beta_r(I_{S_1}^{[\mathbf{n}]})$ is given by (see Theorem 2.7 of
\cite{AC3})
\[
\beta_r(I_{S_1}^{[\mathbf{n}]}) = \sum_{s=0}^{r+1} 
{n-1 \choose s}{n-s \choose r+1-s}; \quad (0 \le r \le n-1).
\]
}
\end{remark}

We now identify standard monomials of 
$\frac{R}{I_{S_a}^{[\mathbf{n}]}}$. Consider the 
following subsets
of the set ${\rm PF}_n$ of parking functions
$\mathbf{p}=(p_1,\ldots,p_n)$ of length $n$.
\begin{enumerate}
\item[(i)] ${\rm PF}_n^1 = \{ \mathbf{p} \in {\rm PF}_n :
p_t \le t,~\forall t ~{\rm and~if }~p_i=i, 
~{\rm then}~ p_j=0~{\rm for~some}~
j \in [i,n]\}$.
\item[(ii)] ${\rm PF}_n^2 = \{ \mathbf{p} \in {\rm PF}_n :
{\rm if }~p_i \ge i, ~{\rm then}~ p_j < j-1 ~{\rm for~all}~
j \in [i+1,n]\}$.
\item[(iii)] ${\rm PF}_n^3 = \{ \mathbf{p} \in {\rm PF}_n :
{\rm if }~p_i \ge i, ~{\rm then}~ p_j < n-(j-i) ~{\rm for~all}~
j \in [i+1,n]\}$.
\end{enumerate}
In view of Lemma \ref{L2}, $\mathbf{x}^{\mathbf{p}} \notin
I_{S_a}^{[\mathbf{n}]}$ if and only if $\mathbf{p} \in
{\rm PF}_n^{a}$ for $1\le a \le 3$. Thus (fine) Hilbert series
$H\left(\frac{R}{I_{S_a}^{[\mathbf{n}]}}, \mathbf{x}\right)$
of $ \frac{R}{I_{S_a}^{[\mathbf{n}]}}$ is given by
$H\left(\frac{R}{I_{S_a}^{[\mathbf{n}]}}, \mathbf{x}\right)
= \sum_{\mathbf{p}\in {\rm PF}_n^{a}} ~\mathbf{x}^{\mathbf{p}}$. In particular, 
$|{\rm PF}_n^{a}| = 
\dim_k\left(\frac{R}{I_{S_a}^{[\mathbf{n}]}} \right)=
H\left(\frac{R}{I_{S_a}^{[\mathbf{n}]}}, \mathbf{1}\right)$, where
$\mathbf{1}=(1,\ldots,1)$. Using the cellular resolution
$\mathbb{F}_{*}(\Delta(\Sigma_n(S_a)))$ supported on the order
complex $\Delta(\Sigma_n(S_a))$, the (fine) Hilber series
$H\left(\frac{R}{I_{S_a}^{[\mathbf{n}]}}, \mathbf{x}\right)$ is
given by
\begin{equation}
H\left(\frac{R}{I_{S_a}^{[\mathbf{n}]}}, \mathbf{x}\right)
= \frac{\sum_{i=0}^n (-1)^i \sum_{(C_1,\ldots,C_i) 
\in {\mathcal{F}_{i-1}^a}} \prod_{q=1}^i \left(
\prod_{j \in C_q \setminus C_{q-1}} x_j^{\mu_{j,C_q}^a} 
\right)}{(1-x_1)\cdots (1-x_n)},
\label{EQ5}
\end{equation}
where $\mathcal{F}_{i-1}^a$ is the set of $i-1$-dimensional
faces of $\Delta(\Sigma_n(S_a))$, $(C_1,\ldots,C_i) \in 
\mathcal{F}_{i-1}^a$ is a (strict) 
chain $C_1 \prec \ldots \prec C_i$
of length $i-1$, $C_0=\emptyset$ and $\mu_{j,C}^a$ is as in
the definition of poset $\Sigma_n(S_a)$.
\begin{proposition}
The number of standard monomials of 
$\frac{R}{I_{S_a}^{[\mathbf{n}]}}$ is given by
\[
\dim_k\left(\frac{R}{I_{S_a}^{[\mathbf{n}]}}\right)
= \sum_{i=1}^n (-1)^{n-i} \sum_{\substack{(C_1,\ldots,C_i)
\in \mathcal{F}_{i-1}^a \\ C_1 \cup \ldots \cup C_i = [n]}}
\prod_{q=1}^i \left(\prod_{j \in C_q \setminus C_{q-1}} 
\mu_{j,C_q}^a \right),
\]
where summation is carried over all $i-1$-dimensional faces
$(C_1,\ldots,C_i) \in \mathcal{F}_{i-1}^a$ of
$\Delta(\Sigma_n(S_a))$ with $C_1 \cup \ldots \cup C_i = [n]$
and $C_0=\emptyset$. Also,
\[
\dim_k\left(\frac{R}{I_{S_a}^{[\mathbf{n}]}}\right)
=   \sum_{\substack{ 0 \le i \le n; \\
(C_1,\ldots,C_i) \in \mathcal{F}_{i-1}^a}} (-1)^{i}
\left( \prod_{q=1}^i (\prod_{j \in C_q \setminus C_{q-1}} 
(\mu_{j,\{j\}}^a-\mu_{j,C_q}^a ))\right) 
\left( \prod_{l \notin C_i} \mu_{l,\{l\}}^a
\right),
\]
where summation is carried over all faces $(C_1,\ldots,C_i) \in \mathcal{F}_{i-1}^a$ 
including the empty face 
$C_0=\emptyset$.
\label{P7}
\end{proposition}
\begin{proof}
As $|{\rm PF}_n^{a}| = 
\dim_k\left(\frac{R}{I_{S_a}^{[\mathbf{n}]}} \right)=
H\left(\frac{R}{I_{S_a}^{[\mathbf{n}]}}, \mathbf{1}\right)$,
letting $\mathbf{x} \to \mathbf{1}$ in the rational function
expression \ref{EQ5} of 
$H\left(\frac{R}{I_{S_a}^{[\mathbf{n}]}}, \mathbf{x}\right)$, and 
applying L'Hospital's rule, we get the first formula. For more
detail, see the proof of Proposition 4.5 of  \cite{AC1}. In order 
to get the second formula, put $y_j = \frac{1}{x_j}$ in 
(\ref{EQ5}) to get a rational function, say
$\tilde{H}\left(\frac{R}{I_{S_a}^{[\mathbf{n}]}}, \mathbf{y}\right)$. 
Now letting $\mathbf{y}\to \mathbf{1}$ in the
product $\left( \prod_{j=1}^n y_j^{\mu_{j,\{j\}}^a -1} 
\right)
\tilde{H}\left(\frac{R}{I_{S_a}^{[\mathbf{n}]}}, \mathbf{y}\right)$,
 we get the second formula, which
is due to Postnikov and Shapiro \cite{PoSh}.
\end{proof}
\begin{theorem}
The number of standard monomials of
$\frac{R}{I_{S_a}^{[\mathbf{n}]}}$ is given by
\[\dim_k\left(\frac{R}{I_{S_a}^{[\mathbf{n}]}} \right)
= |{\rm PF}_n^a | =  \frac{(n+1)!}{2} ,\quad (1 \le a \le 3).\]
\label{T3}
\end{theorem}
\begin{proof} As $\dim_k\left(\frac{R}{I_{S_a}^{[\mathbf{n}]}}
\right)=1$ for $n=1$, we assume that $n >1$. \\
(i) Let $a=1$. Using the second formula
\[
\dim_k\left(\frac{R}{I_{S_1}^{[\mathbf{n}]}}\right)
=   \sum_{\substack{ 0 \le i \le n; \\
(C_1,\ldots,C_i) \in \mathcal{F}_{i-1}^1}} (-1)^{i}
\left( \prod_{q=1}^i (\prod_{j \in C_q \setminus C_{q-1}} 
(\mu_{j,\{j\}}^1-\mu_{j,C_q}^1 ))\right) 
\left( \prod_{l \notin C_i} \mu_{l,\{l\}}^1
\right)
\]
 in Proposition \ref{P7}, we shall show that
\begin{equation}
\dim_k\left(\frac{R}{I_{S_1}^{[\mathbf{n}]}} \right)
= n (n!) + (n-1)((n-1)!) \sum_{\substack{1 \le i \le n;\\
0=j_0 <j_1 < \ldots < j_i < n}} (-1)^i 
\frac{1}{\prod_{q=2}^i j_q}.
\label{EQ6}
\end{equation}
The term corresponding to the empty chain is $n (n!)$. 
Also, for a (strict) chain
$C_1 \prec \ldots \prec C_i$ in $\mathcal{F}_{i-1}^1$, the corresponding
term in the second formula is zero if the chain has a singleton
member. Thus surviving terms are of the form $C_{l}=
[j_{i-l+1},n]$ for some sequence 
$0 = j_0 < j_1 < \ldots < j_i <n$. Note that the term 
corresponding to such a chain is precisely,
$(-1)^i \frac{(n-1)((n-1)!)}{j_2j_3\ldots j_i}$. This proves
(\ref{EQ6}).
Let $\alpha_n = \sum_{i \ge 1} (-1)^{i+1}
\sum_{0=j_0 <j_1 < \ldots < j_i < n} 
\frac{1}{\prod_{q=2}^i j_q}$.  
 Clearly, $\alpha_1=0$. For $n >1$,
we claim that 
$\alpha_n = \frac{n}{2}$. We have,
\begin{eqnarray*}
\alpha_n & = & \sum_{i \ge 1} (-1)^{i+1}
\sum_{0=j_0 <j_1 < \ldots < j_i < n-1} 
\frac{1}{\prod_{q=2}^i j_q} + \sum_{i \ge 1} (-1)^{i+1}
\sum_{0=j_0 <j_1 < \ldots < j_i =n-1} 
\frac{1}{\prod_{q=2}^i j_q} \\
& = & \alpha_{n-1} + \frac{1}{n-1} \sum_{i \ge 2} (-1)^{i+1}
\sum_{0=j_0 <j_1 < \ldots < j_{i-1} < n-1} 
\frac{1}{\prod_{q=2}^{i-1} j_q} + 1 \\
& = & \alpha_{n-1} -\frac{1}{n-1} \alpha_{n-1} + 1
= \frac{n-2}{n-1} \alpha_{n-1}+1. 
\end{eqnarray*}
On solving this recurrence relation, we get 
$\alpha_n = \frac{n}{2}$ for $n >1$. 
Now in view of (\ref{EQ6}),
\[
\dim_k\left(\frac{R}{I_{S_1}^{[\mathbf{n}]}} \right)
= n (n!) + (n-1)((n-1)!) \left(\frac{-n}{2} \right) = \frac{(n+1)!}{2}.
\]
(ii) Let $a=2$. 
As $\dim_k\left(\frac{R}{I_{S_a}^{[\mathbf{n}]}}
\right)=1$ or $3$ for $n=1$ or $2$, respectively, we assume that
$n >2$.
Suppose $\mathcal{F}^2[n] = \cup_{i=1}^n \{
(C_1,\ldots,C_i) \in \mathcal{F}_{i-1}^2 : \cup_{j=1}^i C_j
=[n]\}$. For $\mathcal{C}=(C_1,\ldots,C_i) \in \mathcal{F}^2[n]$, we write 
$\mu^2(\mathcal{C}) = \prod_{q=1}^i \left(\prod_{j \in C_q \setminus C_{q-1}} 
\mu_{j,C_q}^2 \right)$. In view of the first formula in Proposition \ref{P7}, we have
\[ \tilde{\alpha}_n = \dim_k\left(
\frac{R}{I_{S_2}^{[\mathbf{n}]}}\right) = \sum_{\mathcal{C}
\in \mathcal{F}^2[n]} (-1)^{n-\ell(\mathcal{C})-1} 
\mu^2(\mathcal{C}).
\]
Now decompose $\mathcal{F}^2[n]=\mathcal{F}^2[n]' \coprod
\mathcal{F}^2[n]''$, where $\mathcal{C}=(C_1,\ldots,C_i) 
\in \mathcal{F}^2[n]'$ if $|C_1|=1$ and
$\mathcal{C} \in \mathcal{F}^2[n]''$ if $|C_1|=2$. 
Then $\tilde{\alpha}_n= \tilde{\alpha}_n' + \tilde{\alpha}_n''$,
where   
\[\tilde{\alpha}_n' = \sum_{\mathcal{C}
\in \mathcal{F}^2[n]'} (-1)^{n-\ell(\mathcal{C})-1} 
\mu^2(\mathcal{C}) \quad {\rm and}\quad \tilde{\alpha}_n''=\sum_{\mathcal{C}
\in \mathcal{F}^2[n]''} (-1)^{n-\ell(\mathcal{C})-1} 
\mu^2(\mathcal{C}).
\]
A chain $\mathcal{C}=(C_1,\ldots,C_i) \in 
\mathcal{F}^2[n]'$ is called a {\em Type-I, 
Type-II or Type-III } chain, if
$(C_1,C_2)=(\{i\}, \{i,n\})$ for $i<n$, 
$(C_1,C_2)=(\{n\}, \{i,n\})$  for $i < n$ or
$(C_1,C_2)=(\{n\}, \{i,n-1\})$ for $i < n-1$, respectively. Now
\begin{eqnarray*}
\tilde{\alpha}_n' & = & \left[\sum_{\substack{\mathcal{C}\in \mathcal{F}^2[n]'; 
\\ {\rm Type-I}}} + \sum_{\substack{\mathcal{C}\in \mathcal{F}^2[n]'; 
\\ {\rm Type-II}}} + \sum_{\substack{\mathcal{C}\in \mathcal{F}^2[n]'; 
\\ {\rm Type-III}}} \right]
~(-1)^{n-\ell(\mathcal{C})-1}~ \mu^2(\mathcal{C}) \\
& = & n \tilde{\alpha}_{n-1}' -\frac{n}{n-1} \tilde{\alpha}_n''
+n \tilde{\alpha}_{n-1}'' = n \tilde{\alpha}_{n-1}-\frac{n}{n-1}
\tilde{\alpha}_{n}''.
\end{eqnarray*}
Claim : $\tilde{\alpha}_n'' = - \frac{(n-1) (n!)}{2}$.\\
For $1 \le t \le n-1$,
consider saturated chains $\mathcal{C}^{(t)}$
 in $\mathcal{F}^2[n]''$  of the form 
 \[
 \mathcal{C}^{(t)} : \{t,n\} \prec \{t,n-1\} \prec \ldots
 \prec \{t,t+1\} \prec \{t-1,t\} \prec \ldots \prec \{1,2\} .
 \]
 Then $\mu^2(\mathcal{C}^{(t)}) = t ((n-1)!)$. 
Any other chain in  $\mathcal{F}^2[n]''$ is either of the form
 \[
 \mathcal{C}: \{r,n\} \prec \ldots \prec \{r,r+1\} \prec \ldots \prec  \{s-1,s\} \prec
 \{l,s-1\} \prec \{l,s-2\} \prec
 \ldots \prec \ldots
 \]
or
 \[
 \mathcal{C}': \{r,n\} \prec \ldots \prec \{r,r+1\} \prec \ldots \prec \{s'-1,s'\}
  \prec \{l',s'-2\} \prec \ldots \prec \ldots , \quad  ({\rm for}~3 \le r \le n-1),
 \]
where $s$ (or $s'$) is the largest integer such that 
$\{l,s-1\}$ covers
 $\{s-1,s\}$ in $\mathcal{C}$ (or $\{l',s'-1\}$ is
 not in $\mathcal{C}'$) for some $l < s-2$ (or $l' < s'-2$).
Let $\tilde{\mathcal{C}} =
 \mathcal{C}\setminus \{\{l,s-1\}\}$ be the chain obtained from
$\mathcal{C}$ on deleting  $\{l,s-1\}$ and 
 $\tilde{\mathcal{C}}' = \mathcal{C}' \cup \{\{l',s'-1\}\}$
be the chain obtained from 
$\mathcal{C}'$ on adjoining $\{l',s'-1\}$.
Clearly,
 $\mu^2(\mathcal{C}) = \mu^2(\tilde{\mathcal{C}})$ and
 $\mu^2(\mathcal{C}') = \mu^2(\tilde{\mathcal{C}}')$. As length
 $\ell(\mathcal{C}) = \ell(\tilde{\mathcal{C}})+1$
 and $\ell(\mathcal{C}')=\ell(\tilde{\mathcal{C}}')-1$, the
 terms in 
 $\tilde{\alpha}_n''  =  \sum_{\mathcal{C}
\in \mathcal{F}^2[n]''} (-1)^{n-\ell(\mathcal{C})-1} 
\mu^2(\mathcal{C})$ corresponding to chains $\mathcal{C}
\in \mathcal{F}^2[n]''$ different from $\mathcal{C}^{(t)}$ 
cancel out. Thus
 \[
 \tilde{\alpha}_n''  = \sum_{t=1}^{n-1} 
 (-1)^{n-\ell(\mathcal{C}^{(t)})-1}~ \mu^2(\mathcal{C}^{(t)})
 =  \sum_{t=1}^{n-1} (-1)^{n-(n-2)-1}~ t((n-1)!) = 
- \frac{(n-1)(n!)}{2}.
\]
 Now $\tilde{\alpha}_n = \tilde{\alpha}_n'+\tilde{\alpha}_n''=
 n\tilde{\alpha}_{n-1} -\frac{n}{n-1} \tilde{\alpha}_n'' + 
 \tilde{\alpha}_n'' = n \tilde{\alpha}_{n-1} + \frac{n!}{2}$. 
 On solving this recurrence, we get 
 $\tilde{\alpha}_n = \frac{(n+1)!}{2}$, as desired.\\
 (iii) Let $a=3$ and assume $n >2$. Proceeding as in part(ii), we 
 write
 \[  \dim_k\left(
\frac{R}{I_{S_3}^{[\mathbf{n}]}}\right) = \sum_{\mathcal{C}
\in \mathcal{F}^3[n]} (-1)^{n-\ell(\mathcal{C})-1} 
\mu^3(\mathcal{C}),
\]
where  $\mathcal{F}^3[n]$ is the collection of all chains
$\bar{\mathcal{C}}=(C_1,\ldots,C_i)$  in $\mathcal{F}_{i-1}^3$ (for some $i$) with
$\cup_{j=1}^i C_j =[n]$ and 
 $\mu^3(\bar{\mathcal{C}}) = \prod_{q=1}^i \left(\prod_{j \in C_q \setminus C_{q-1}} 
\mu_{j,C_q}^3 \right)$. For $1 \le t \le n-1$, let 
$\bar{\mathcal{C}}^{(t)}$ be the chain in $\mathcal{F}^3[n]$ of 
the form
\[
\bar{\mathcal{C}}^{(t)} : \{t\} \prec \{t,t+1\} \prec \ldots
\prec \{t,n-1\} \prec \{t, n\} 
\prec \{t-1,n\} \prec \ldots \prec \{1,n\}
\]
 and $\bar{\mathcal{C}}^{(t)}\setminus\{\{t\}\}$ is the chain 
 obtained from $\bar{\mathcal{C}}^{(t)}$ by deleting the first element
 $\{t\}$. Now $\mu^3(\bar{\mathcal{C}}^{(t)}) = n!$
 and $\mu^3(\bar{\mathcal{C}}^{(t)} \setminus \{\{t\}\}) = 
 t ((n-1)!)$.
There is one more chain $\bar{\mathcal{C}} : \{n\} 
\prec \{n-1,n\} \prec \ldots \prec \{1,n\}$ in 
$\mathcal{F}^3[n]$, with $ \mu^3(\bar{\mathcal{C}}) = n!$.
  As in part (ii), it can be shown that the terms
 corresponding to remaining chains cancel out. Thus
 \[
 \dim_k\left(
\frac{R}{I_{S_3}^{[\mathbf{n}]}}\right) = n(n!) - 
(1+2+\ldots+ (n-1))((n-1)!) = \frac{(n+1)!}{2} .
\]
\hfill $\square$
\end{proof}
Theorem \ref{T3} shows that the integer sequence
$\left\lbrace \dim_k\left(\frac{R}{I_{S_a}^{[\mathbf{n}]}}\right)
=\frac{(n+1)!}{2} \right\rbrace_{n=1}^{\infty} $ 
for $1 \le a \le 3$ is
the integer sequence (A001710) in OEIS \cite{S}. As
$|{\rm PF}_n^a| = \frac{(n+1)!}{2}$, it is expected that the set
${\rm PF}_n^a$ could be easily enumerated. Let $\mathbf{p} \in 
{\rm PF}_n^1$. Then $p_t \le t;~\forall t$ and $p_i = i$ implies
that $p_j = 0$ for some $j \in [i+1,n]$. We count $\mathbf{p} 
\in {\rm PF}_n^1$ according to the value $s$ of the 
largest $t \in [n]$ with $p_t=t$.  If $p_t  < t; 
\forall t \in [n]$, then we take $s=0$. As $p_n <n$, we have
$0 \le s \le n-1$. For $s=0$, any 
$\mathbf{p}=(p_1,\ldots,p_n) \in \mathbf{N}^n$ such that
 $p_t < t; ~\forall t$ is a parking function and number
 of such $\mathbf{p} \in {\rm PF}_n^1$ is precisely
 $\prod_{t=1}^{n} (t) = n!$. Now let $s \ge 1$. Any 
 sequence $\mathbf{p}=(p_1,\ldots,p_n) \in \mathbb{N}^n$ 
 satisfying  conditions
 \begin{equation}
 p_t \le t ~\forall t <s, ~~ p_s=s, ~~{\rm and}~ p_j < j ~ 
 \forall  j > s,  {\rm with~ at~ least~ one}~ p_j = 0,
 \label{EQ7}
 \end{equation} 
 is always a parking function. The number of $\mathbf{p}$ satisfying 
 conditions (\ref{EQ7}) is 
 \[
 \prod_{t=1}^{s-1} (t+1) \left[ \prod_{j=s+1}^n j -
 \prod_{j'=s+1}^n (j'-1) \right] = (n-s) ((n-1)!).
 \]
 This shows that $|{\rm PF}_n^1| = \sum_{s=0}^{n-1} (n-s) ((n-1)!)
 = \frac{(n+1)!}{2}$. Similarly, ${\rm PF}_n^a$
 for $a=2,3$ can also be enumerated. However, it is still
 an interesting problem to construct an (explicit) bijection
 $\phi : {\rm PF}_n^a \longrightarrow F_{n+1}(21)$, where 
 $F_{n+1}(21)$ is the set of rooted-labelled increasing 
 forests on $[n+1]$.
 
{\sc Case-2} : To monomial ideals $I_{T_1}^{[\mathbf{n}]}$ and
$I_{T_2}^{[\mathbf{n}]}$, we  
associate finite posets $\Sigma_n(T_1)$ and 
$\Sigma_n(T_2)$ respectively, as below.
\begin{enumerate}
\item[{\rm (i)}] Let $\Sigma_n(T_1) =
\{ \{\ell\}, \{i, n\} : 1 \le \ell \le n-1; ~1 \le i \le n\}$, where $\{n,n\}=\{n\}$.
We define a poset structure on $\Sigma_n(T_1)$ by describing
cover relations. For $\ell, \ell' \in [n-1]$ and $i,i' \in [n]$,
$\{\ell\}$ covers $\{\ell'\}$, if $\ell'=\ell+1$. 
Also, $\{i,n\}$ covers $\{\ell'\}$
(or $\{i',n\}$) if $i=\ell'$ (respectively, $i'=i+1$). 
The monomial labels $\omega_{\{\ell\}}=x_{\ell}^{\ell+1},
~\omega_{\{n\}}=x_n^n$ and 
$\omega_{\{i,n\}}= x_i^i x_{n}^i$ for $1 \le \ell, i < n$. 
Set $\hat{\mu}_{j,C}^{1}$
 for $C \in \Sigma_n(T_1)$ so that 
 $\omega_C = \prod_{j \in C} x_j^{\hat{\mu}_{j,C}^{1}}$.
 \item[{\rm (ii)}] Let $\Sigma_n(T_2) = \Sigma_n(T_1)$. But
 the poset structure on $\Sigma_n(T_2)$ is obtained 
 by interchanging $\{i\}$ with $\{n-i\}$ (and also,
  $\{i,n\}$ with
 $\{n-i,n\}$)(for $1 \le i < n$)
 in the poset $\Sigma_n(T_1)$. The cover relations of the poset
$\Sigma_n(T_2)$ are given as follows.
For $\ell, \ell',i,i' \in [n-1]$,
$\{\ell\}$ covers $\{\ell'\}$, if $\ell'=\ell-1$ and 
$\{i,n\}$ covers $\{\ell'\}$
(or $\{i',n\}$) if $i=\ell'$ (respectively, $i'=i-1$). 
In addition, $\{1,n\}$ covers $\{n\}$.
The monomial labels $\omega_{\{\ell\}}=x_{\ell}^{n-\ell+1},
~\omega_{\{n\}}=x_n^n$ and 
$\omega_{\{i,n\}}= x_i^{n-i} x_{n}^{n-i}$ for $1 \le \ell, i < n$. 
Set $\hat{\mu}_{j,C}^{2}$
 for $C \in \Sigma_n(T_1)$ so that 
 $\omega_C = \prod_{j \in C} x_j^{\hat{\mu}_{j,C}^{2}}$.
\end{enumerate} 

The Hasse diagram of $\Sigma_4(T_1)$ 
and $\Sigma_4(T_2)$ are given in {\sc Figure-3}. 
\begin{figure}[!hbpt]
{\begin{tikzpicture}
\node (max) at (0,2) {$14$};
\node (a) at (-1.5,0.5) {$1$};
\node (c) at (1.5,0.5) {$24$};
\node (d) at (-1.5,-1) {$2$};
\node (g) at (-1.5,-2.5) {$3$};
\node (b) at (1.5,-2.5) {$4$}; 
\node (f) at (1.5,-1) {$34$};
\node (h) at (0,-3.5) {$\Sigma_4(T_1)$};
\draw (d) -- (a) -- (max)(c) -- (max)(f) -- (c)(d) -- (g)(f) -- (b)(d) -- (c)(g) -- (f);
\end{tikzpicture}}
\hspace{8em}
{\begin{tikzpicture}
\node (max) at (9,2) {$34$};
\node (a) at (7.5,0.5) {$3$};
\node (c) at (10.5,0.5) {$24$};
\node (d) at (7.5,-1) {$2$};
\node (g) at (7.5,-2.5) {$1$};
\node (b) at (10.5,-2.5) {$4$}; 
\node (f) at (10.5,-1) {$14$};
\node (h) at (9,-3.5) {$\Sigma_4(T_2)$};
\draw (g) -- (d) -- (a) -- (max) -- (c) -- (f) -- (b)(g) -- (f)(d) -- (c);
\end{tikzpicture}}
\caption{}
\end{figure}

\begin{proposition}
{\rm (i)}. ~The ideals $I_{T_1}^{[\mathbf{n}]}$ 
and $I_{T_2}^{[\mathbf{n}]}$ are order
monomial ideals. \\
{\rm (ii)}.~The free complex 
$\mathbb{F}_{*}(\Delta(\Sigma_n(T_b)))$ is the minimal 
cellular resolution of $I_{S_a}^{[\mathbf{n}]}$ supported on the 
order complex $\Delta(\Sigma_n(T_b))$ for $1 \le b \le 2$. Thus
the $r^{th}$ Betti number $\beta_r(I_{T_b}^{[\mathbf{n}]})$ is given by 
\[ \beta_r(I_{T_b}^{[\mathbf{n}]}) =  {n \choose r+1} +
(r+1) {n-1 \choose r+1} + r {n-1 \choose r},  
\quad (1 \le r \le n-1).
\]
\label{P8}
\end{proposition}
\begin{proof}
From the definitions of the poset $\Sigma_n(T_b)$, it is 
clear that the ideal $I_{T_b}^{[\mathbf{n}]}$ is
an order monomial ideal. Further, the cellular resolution
$\mathbb{F}_{*}(\Delta(\Sigma_n(T_b)))$ is the minimal 
resolution of $I_{S_a}^{[\mathbf{n}]}$ supported on the 
order complex $\Delta(\Sigma_n(T_b))$ because monomial label
on any face of $\Delta(\Sigma_n(T_b))$ is different from the
monomial label on subfaces. Thus the $r^{th}$ Betti number
$\beta_r(I_{T_b}^{[\mathbf{n}]})$ equals the number 
(strict) chains of
length $r$ in the poset $\Sigma_n(T_b)$. Since $\Sigma_n(T_2)$
is obtained from $\Sigma_n(T_1)$ by changing $i$ to $n-i$ for
$i \in [n]$, number of chains of length $r$ in both the posets
are same. We count chains of length $r$ in $\Sigma_n(T_1)$ for 
$0 \le r \le n-1$. Consider a (strict) chain 
\[ \mathcal{C} : 
C_1 \prec C_2 \prec \ldots \prec C_s \prec C_{s+1} \prec 
\ldots \prec C_{r+1} .
\]
If all $C_j$ are of the form $\{t_j,n\}$ for $t_j \in [n]$, then the
chain $\mathcal{C}$ can be identified with a $r+1$-subset 
$\{t_1,\ldots,t_{r+1}\}$ of
$[n]$. Thus number of such chains is ${n \choose r+1}$. If 
$C_s = \{ t_s \}$ and $C_{s+1}=\{t_{s+1}, n \}$ for some
$s$ with $t_{s+1} <  t_s$, then the chain $\mathcal{C}$ can be 
identified with a $r+1$-subset
$\{t_1,\ldots,t_{r+1}\}$
 of $[n-1]$ with a chosen element $t_s$. Any $j \in
 \{t_1,\ldots,t_{r+1}\}$ represent singleton $\{j\}$ if $j \ge 
 t_s$, while it represent $\{j,n\}$ for $j < t_s$. The number of
 such chains is  precisely $(r+1)~{n-1 \choose r+1}$. Now we 
 count chains $\mathcal{C}$
 with $C_s =\{t_s\}$ and $C_{s+1}= \{t_s,n \}$
 (i.e., $t_s=t_{s+1}$). In this case, chain 
 $\mathcal{C}$ can be identified with a $r$-subset
 $\{t_1,\ldots,t_s=t_{s+1},\ldots,t_{r+1}\}$ of $[n-1]$ with
 a chosen element $t_s$. Thus number of such chains is
 $r~{n-1 \choose r}$. Since any $r$-chain $\mathcal{C}$ in 
 $\Sigma_n(T_1)$ is a chain of one of the three types, 
  we get the desired result.
 \hfill $\square$  
\end{proof}

Consider the following subsets of ${\rm PF}_n$
of parking function $\mathbf{p}=(p_1,\ldots,p_n)$.
\begin{enumerate}
\item[(i)] $\widehat{{\rm PF}}_n^1 = \{ \mathbf{p} \in {\rm PF}_n
: p_t \le t,~\forall t ~{\rm and~if}~p_i=i,~{\rm then}~ p_n <i \}$.
\item[(ii)] $\widehat{{\rm PF}}_n^2 = \{ \mathbf{p} \in {\rm PF}_n
: p_{n-t} \le t,~\forall t ~{\rm and~if}~p_{n-i}=i,~{\rm then}~ p_n <i \}$.
\end{enumerate}
In view of Lemma \ref{L1}, $\mathbf{x}^{\mathbf{p}}
\notin I_{T_b}^{[\mathbf{n}]}$ if and only if $ \mathbf{p}
\in \widehat{{\rm PF}}_n^b$ for $b=1,2$. Thus, 
$|\widehat{{\rm PF}}_n^b | = \dim_k\left(
\frac{R}{I_{T_b}^{[\mathbf{n}]}}\right)$. Also, the mapping
$(p_1,p_2,\ldots,p_{n-1},p_n) \mapsto 
(p_{n-1},p_{n-2},\ldots,p_1,p_n)$ induces a bijection between
$\widehat{{\rm PF}}_n^1$ and $\widehat{{\rm PF}}_n^2$.
\begin{theorem}
The number of standard monomials of 
$\frac{R}{I_{T_b}^{[\mathbf{n}]}}$ is given by
\[ |\widehat{{\rm PF}}_n^b | = \dim_k\left(
\frac{R}{I_{T_b}^{[\mathbf{n}]}}\right) = s(n+1,2); \quad 
(b=1,2),
\]
where $s(n+1,2)$ is the {\rm (}signless{\rm )} Stirling number of 
the first kind.
\label{T4}
\end{theorem}
\begin{proof}
We take $b=1$. 
Proceeding as in Proposition \ref{P7}, we get
\[ \dim_k\left(
\frac{R}{I_{T_1}^{[\mathbf{n}]}}\right) = \sum_{\mathcal{C}
\in \widehat{\mathcal{F}}^1[n]} (-1)^{n-\ell(\mathcal{C})-1}
~\widehat{\mu}^1(\mathcal{C}) ,
\]
where $\widehat{\mathcal{F}}^1[n]$ is the collection 
of all chains $\mathcal{C}=(C_1,\ldots,C_i)$ in 
$\Sigma_n(T_1)$ such that $C_1 \cup \ldots \cup C_i = [n]$
and 
$\widehat{\mu}^1(\mathcal{C}) = 
\prod_{q=1}^i \left(\prod_{j \in C_q \setminus C_{q-1}} 
\widehat{\mu}_{j,C_q}^1 \right)$. For $1 \le t \le n$, let 
$\widehat{\mathcal{C}}^{(n)} : \{n\} \prec \{n-1,n\} \prec \ldots \prec \{1,n\}$,
\[\widehat{\mathcal{C}}^{(t)} : \{n-1\} \prec \ldots \prec
\{t\} \prec \{t,n\} \prec \{t-1,n\} \prec \ldots \prec \{1,n\}; \quad ( 1 \le t \le n-1)
\]
and $\widehat{\mathcal{C}}^{\prime(t)}$ be the chain 
obtained from
$\widehat{\mathcal{C}}^{(t)}$ on deleting $\{t,n\}$. For
$t=n$, we have $\{n,n\} = \{n\}$. It is clear
that $\widehat{\mathcal{F}}^1[n] = 
\{ \widehat{\mathcal{C}}^{(t)}, 
\widehat{\mathcal{C}}^{\prime(t)} : 1 \le t \le n \}$. Also,
$\widehat{\mu}^1(\widehat{\mathcal{C}}^{(t)}) = n!$ and
$\widehat{\mu}^1(\widehat{\mathcal{C}}^{\prime(t)}) =
 \frac{t-1}{t} (n!)$ for $1 \le t \le n$. As 
$\ell(\widehat{\mathcal{C}}^{(t)}) =
\ell(\widehat{\mathcal{C}}^{\prime(t)}) + 1 = n-1$, we see that 
\begin{eqnarray*}
\dim_k\left(
\frac{R}{I_{T_1}^{[\mathbf{n}]}}\right) & = & \sum_{t=1}^n \left(
\widehat{\mu}^1(\widehat{\mathcal{C}}^{(t)})
- \widehat{\mu}^1(\widehat{\mathcal{C}}^{\prime(t)}) \right)
=\sum_{t=1}^n \left( n! - \frac{t-1}{t} n!\right) \\
& = &  \sum_{t=1}^n \frac{n!}{t}
 =  \left(1+\frac{1}{2} +\ldots+\frac{1}{n}\right) n! = 
 s(n+1,2).
\end{eqnarray*}
\hfill $\square$
\end{proof}
A nice formula $|\widehat{{\rm PF}}_n^1 |=
|\widehat{{\rm PF}}_n^2 |=s(n+1,2)$, 
deserves a combinatorial proof.
We count parking functions
$\mathbf{p}=(p_1,\ldots,p_n)$ in $\widehat{{\rm PF}}_n^1$
 according to the value of $p_n$. Clearly, $0 \le p_n \le n-1$. 
 For any $0 \le t \le n-1$, we see that $p_n=t$ implies that
 $p_i < i $ for all $i \le t$ and $p_j \le j$ for $j > t$. Also,
 any $(p_1,\ldots,p_n)$ with $p_n=t$ and $p_i < i$ 
 for all $ i \le t$, while
  $p_j \le j $ for all $t<j\le n-1$ is always a 
  parking function of length $n$. Thus number of 
  $\mathbf{p} = (p_1,\ldots,p_n) \in 
  \widehat{{\rm PF}}_n^1$ with $p_n=t$ is 
  $\left(\prod_{i=1}^t i\right)
  \left(\prod_{j=t+1}^{n-1} (j+1)\right) = \frac{n!}{t+1}$.
  Hence, $|\widehat{{\rm PF}}_n^1| = 
  \sum_{t=0}^{n-1} \frac{n!}{t+1}$.
  
Theorem \ref{T4} shows that the integer sequence 
$\left\lbrace \dim_k\left(
\frac{R}{I_{T_b}^{[\mathbf{n}]}}\right) = 
s(n+1,2) \right\rbrace_{n=1}^{\infty}$ for $b=1,2$ is the integer
sequence (A000254) in OEIS \cite{S}. 

{\sc Case-3} : We finally consider the monomial ideal
 $I_{U}^{[\mathbf{n}]}$. The minimal generators 
 $\prod_{j \in A} x_j^{\nu_{j,A}}$ of 
 $I_{U}^{[\mathbf{n}]}$ are parametrized by the 
 poset $\Sigma_n$. Again, it is straight forward 
 to verify that the ideal 
 $I_{U}^{[\mathbf{n}]}$ is an order monomial ideal and the
 cellular resolution $\mathbb{F}_{*}(\Delta(\Sigma_n))$ 
 supported on the order 
 complex $\Delta(\Sigma_n)$ is the minimal free resolution
 of $I_{U}^{[\mathbf{n}]}$. Thus $r^{th}$ Betti number
 $\beta_r(I_{U}^{[\mathbf{n}]}) = (r!) S(n+1,r+1)$ for
 $0 \le r \le n-1$.
 
 Now we describe standard monomials of 
 $\frac{R}{I_{U}^{[\mathbf{n}]}}$. Let $\overline{{\rm PF}}_n
 = \{ \mathbf{p}\in {\rm PF}_n : 
 \mathbf{x}^{\mathbf{p}} \notin I_{U}^{[\mathbf{n}]} \}$.
 \begin{lemma}
 Let $\mathbf{p}=(p_1,\ldots,p_n) \in {\rm PF}_n$. Then $\mathbf{p} \in 
 \overline{{\rm PF}}_n$ if and only if, there exists a 
 permutation $\alpha \in \mathfrak{S}_n$ such that 
 $p_{\alpha_i} < ~\nu_{\alpha_i,T_i}$ for all $i$, where
 $\alpha_i = \alpha(i)$, $T_1=[n]$ and $T_j = [n] \setminus 
 \{\alpha_1,\ldots,\alpha_{j-1}\}$ for $j \ge 2$. Also,
 $\nu_{j,T}$ is in the Lemma \ref{L1}.
 \label{L3}
 \end{lemma}
 \begin{proof}
 Proof is similar to the proof of Theorem 4.3 of \cite{AC1}.
 \hfill $\square$
 \end{proof}
 Proceeding as in Proposition \ref{P7}, we get a 
 combinatorial formula for the number of standard monomials of 
 $\frac{R}{I_{U}^{[\mathbf{n}]}}$.
 \begin{proposition}
 The number of standard monomials of 
$\frac{R}{I_{U}^{[\mathbf{n}]}}$ is given by
\[
|\overline{{\rm PF}}_n| = \dim_k\left(\frac{R}{I_{U}^{[\mathbf{n}]}}\right)
= \sum_{i=1}^n (-1)^{n-i} \sum_{\emptyset=C_0 
\subsetneq C_1 \subsetneq \ldots \subsetneq C_i =[n]}
\prod_{q=1}^i \left(\prod_{j \in C_q \setminus C_{q-1}} 
\nu_{j,C_q} \right),
\]
where summation is carried over all strict chains
$\emptyset=C_0 \subsetneq  C_1 \subsetneq \ldots
\subsetneq C_i =[n] $.
 \label{P9}
 \end{proposition}
 
 Neither using Proposition \ref{P9}, nor by any combinatorial tricks, we could determine 
 $|\overline{{\rm PF}}_n| = \dim_k\left(\frac{R}{I_{U}^{[\mathbf{n}]}}\right)$.
Thus, we ask the following question.
 
 {\bf Question} : Is it possible to identify the sequence
 $\left\lbrace \dim_k\left(\frac{R}{I_{U}^{[\mathbf{n}]}}\right)
 \right\rbrace_{n = 1}^{\infty}$ with some well known combinatorially interesting integer sequence?
 
 Computations for smaller values of $n$
 suggest
 that this integer sequence could be
  (A003319) in OEIS \cite{S}.
 
 {\bf Acknowledgements} : The second author is thankful to  
 CSIR, Government of India for financial support.

\end{document}